\documentclass[10pt,reqno]{amsart}
\usepackage{amssymb,amsfonts,amsthm,amscd,stmaryrd,dsfont,esint,upgreek,constants,todonotes}
\usepackage[mathcal]{euscript}
\usepackage[latin1]{inputenc}   % pour les accents
\usepackage[mathcal]{euscript}
\usepackage{graphicx,color}
\numberwithin{equation}{section}
\newcommand{\N}{\mathbb N}
\newcommand{\R}{\mathbb R}
\def\E{\mathbb E}
\def\P{\mathbb P}
\newcommand{\Pw}{\mathcal P_2(\R^d)}
\newcommand{\Pk}{\mathcal P_1(\R^d)}
\newcommand{\ms}{\mathcal}

\newconstantfamily{C}{symbol=C}
\newconstantfamily{c}{symbol=c}
\newconstantfamily{O}{symbol=\Omega}

\def\XXint#1#2#3{{\setbox0=\hbox{$#1{#2#3}{\int}$}
\vcenter{\hbox{$#2#3$}}\kern-.5\wd0}}

\numberwithin{equation}{section}
\newtheorem{thm}{Theorem}[section]
\newtheorem{lem}[thm]{Lemma}
\newtheorem{cor}[thm]{Corollary}
\newtheorem{prop}[thm]{Proposition}

\theoremstyle{definition}
\newtheorem{defn}[thm]{Definition}
\newtheorem{rmk}[thm]{Remark}

\def\smallnegint{\mathop{\int\mkern-13mu
        \raise.5ex\hbox{${\scriptscriptstyle\diagup}$}}\nolimits}
\def\ds{\displaystyle}

\def\ep{\varepsilon}

\def\bx{{\bf x}}
\def\ssetminus{\,\raise.4ex\hbox{$\scriptstyle\setminus$}\,}
\def \lg{\langle}
\def \rg{\rangle}
\newcommand{\be}{\begin{equation}}
\newcommand{\ee}{\end{equation}}

\newcommand{\bc}{\begin{case}}
\newcommand{\ec}{\end{cases}}
\newcommand{\bs}{\begin{split}}
\newcommand{\es}{\end{split}}

\newcommand{\vs}{\vskip.075in}

\renewcommand{\bar}{\overline}
\renewcommand{\tilde}{\widetilde}
\renewcommand{\hat}{\widehat}

\def\dk{{\bf d}_1}

\begin{document}
\title[Mean field control]{Regularity of the value function and quantitative propagation of chaos for mean field control problems}

\author{Pierre Cardaliaguet$^{(1)}$ and Panagiotis E Souganidis$^{(2)}$}
%\author[Pierre Cardaliaguet and Panagiotis E. Souganidis]
%{Pierre Cardaliaguet and Panagiotis E. Souganidis}
%\address{Universit\'e Paris-Dauphine, PSL Research University, Ceremade, 
%Place du Mar\'echal de Lattre de Tassigny, 75775 Paris cedex 16 - France}
%\email{cardaliaguet@ceremade.dauphine.fr }
%\address{Department of Mathematics, University of Chicago, Chicago, Illinois 60637, USA}
%\email{souganidis@math.uchicago.edu}
%\vskip-0.5in 
%\thanks{\hskip-0.149in The second author was partially supported by the National Science Foundation grants DMS-1266383 and DMS-1900599, the Office for Naval Research grant N000141712095 and the Air Force Office for Scientific Research grant FA9550-18-1-0494.}

\dedicatory{Version: \today}

%\tableofcontents

\begin{abstract} We investigate a mean field optimal control problem obtained in the limit of the optimal control of large particle systems {with forcing and terminal data which are not assumed to be convex.} We prove that the value function,  which is known to be Lipschitz continuous but not of class $C^1$, in general,  {without convexity}, is actually smooth in an open and dense subset of the space of times and probability measures.  As a consequence, we prove a new quantitative propagation of chaos-type result  for the optimal solutions of the particle system starting from this open and dense set.  
\end{abstract}

\maketitle

\section*{Introduction} 

The paper is about the regularity of the value function  and quantitative propagation of chaos for mean field control (MFC for short) problems obtained as the limit of optimal control problems for  large particle systems {with forcing and terminal data which are not assumed to be convex.}. The value function of MFC problems is known to be a Lipschitz continuous but, in general,  not $C^1-$function in the space of  time and probability measures. Our first result is that there exists an open and dense subset of time and probability measures where 
the value function is actually smooth. The second result is a new quantitative propagation of chaos-type property  for the optimal solutions of the particle system starting from this open and  dense set. 
\vs
 
\subsection*{The background} In order to state the results it is necessary to introduce the general set-up. 
\vs

We consider the problem of controlling optimally $N$ particles in order to minimize a criterion of the form
\be\label{defJN}
J^N(t_0,{\bf x}_0, \alpha) :=\E\left[ \frac{1}{N} \int_{t_0}^T\sum_{i=1}^N L(X^i_t, \alpha^i_t)dt + \int_{t_0}^T \mathcal F(m^N_{{\bf X}_t})dt + \mathcal  G(m^N_{{\bf X}_T})\right].
\ee
Here $T>0$ is a fixed time horizon, $t_0\in [0,T]$ and $\bx_0=(x^1_0,\dots, x^N_0)\in (\R^d)^N$ are respectively  the initial time and the initial position of the system at time $t_0$. The minimization is over the set $\ms A^N$  of admissible controls $\alpha= (\alpha^k)_{k = 1}^N$ in $L^2([0,T] \times \Omega ; (\R^d)^N)$ which are adapted to the filtration generated by the independent $d-$dimensional Brownian motions $(B^i)_{i=1, \dots, N}$,  and  the trajectories 
${\bf X}= (X^1, \dots, X^N)$ satisfy, for each $k\in \{1, \dots, N\}$, 
$$
X^k_t = x^k_0+\int_{t_0}^t \alpha^k_s ds +\sqrt{2}(B^k_t - B^k_{t_0}) \ \ \text {for} \ \  t\in [t_0,T]. 
$$
In  \eqref{defJN}, $m^N_{{\bf X}_t}$ is the empirical measure of the process  ${\bf X}_t$ given by 
\be\label{m}
m^N_{{\bf X}_t}:= \dfrac{1}{N}\sum_{k=1}^N \delta_{X^k_t},
\ee
where $\delta_x$ is the Dirac mass at $x$.
\vs
The maps $\mathcal F$ and  $\mathcal G$, which are  defined on (suitable subsets of) the set of Borel probability measures on $\R^d$, couple all the particles $X^i$.  Finally, the cost function $L=L(x,\alpha):\R^d\times \R^d\to \R$ is, typically, convex and grows quadratically  in the second variable. 
\vs

The value function for this optimization problem reads 
\be \label{def.VN}
\begin{split}
\mathcal V^N(t_0, \bx_0)&:= \inf_{\alpha \in \mathcal A^N} J^N(t_0,{\bf x}_0, \alpha)\\
&=\inf_{\alpha \in \mathcal A^N} 
 \E\left[ \int_{t_0}^T (\frac{1}{N}\sum_{k=1}^N L(X^k_t,\alpha^k_t) +\mathcal F(m^N_{{\bf X}_t}))dt +
\mathcal G(m^N_{{\bf X}_T}) \right].
\end{split}
\ee

\vs
In a more general framework and under slightly different conditions on the data,  Lacker \cite{La17} proved that the empirical measure $m^N_{{\bf X}^N_t}$ associated to the optimal trajectories of \eqref{def.VN} converges in a suitable sense to the (weak) optimal solution of the mean field control problem (written here in a strong sense) consisting in minimizing the quantity  
\be\label{defJinfty}
J^\infty(t_0, m_0, \alpha) = \E\left[ \int_{t_0}^T L(X_t, \alpha_t)dt + \int_{t_0}^T \mathcal F({\mathcal L}(X_t))dt + \mathcal  G(\mathcal L (X_T))\right],
\ee
where $m_0$ is an initial distribution of the particles at time $t_0$, $\alpha \in \ms A$, the set of admissible controls consisting of square integrable $\R^d-$valued processes adapted to a Brownian motion $B$ and to an initial condition $\bar X_0$, which is  independent of $B$ and of law $m_0$,  and  the process $(X_t)_{t\in [t_0,T]} $ satisfies
$$
X_t= \bar X_0 +\int_{t_0}^t \alpha_sds + \sqrt{2}(B_t-B_{t_0}) \ \ \text{for} \ \ t\in [t_0,T],
$$
and $\ms L(X_t)$ denotes the law of $X_t$. 
\vs

The value function $\ms U$ of the last optimization problem is given (heuristically at this stage) by 
\be\label{takis1}
\begin{split}
\mathcal U(t_0,m_0):&= \inf_{\alpha \in \ms A} J^\infty(t_0, m_0, \alpha) \\
&=\inf_{\alpha\in \ms A} \E[\int_{t_0}^T\big( L(X_t, \alpha_t) + \ms F( \ms L(X_t)\big) dt + \ms G(\ms L(X_T)].
\end{split}
\ee
 
\vs

In addition,  \cite{La17} points out that  there is a propagation of chaos-type property (an easy consequence of Sznitman characterization of propagation of chaos),  if the minimization problem  \eqref{takis1} has a unique weak minimizer. 
Note, however, that such uniqueness is known only when the maps $(x,\alpha)\to L(x,\alpha)$,  $m\to \mathcal F(m)$ and $m\to \mathcal G(m)$ are globally convex. The conclusions  of \cite{La17} were extended to problems with common noise in Djete, Possama\"i and Tan \cite{DPT20}. Several other results on the convergence of MFC problems without diffusion were obtained  in Cavagnari, Lisini, Orrieri and  Savar\'{e} \cite{CaLiOrSa} and  Gangbo, Mayorga and Swiech \cite{GMS}.  A quantitative convergence rate for the value function $\mathcal V^N$ to $\mathcal U$ was given, for problems on a finite state space, in Kolokoltsov \cite{Ko12} and Cecchin \cite{Ce21} and, for problems on the continuous state space, in Baryaktar and Chakraborty \cite{barcha21} under a certain structural dependence of the data on the measure variable, in Germain, Pham and Warin \cite{gpw21} under the assumption that the limit value is smooth, and in Cardaliaguet, Daudin, Jackson and Souganidis \cite{CDJS} under a decoupling assumption on the Hamiltonian. In addition, a propagation of chaos is proved in \cite{barcha21, Ce21, gpw21} assuming, however, that the limit value function is smooth.
\vs
\subsection*{The results} 
In this paper we study nonconvex MFC problems  for which the limit value function is not expected to be globally smooth  and show that for a  large class (a dense and open set) of initial times and measures the value function $\ms U$ is  smooth (Theorem~\ref{thm.mainreguU0}) and the propagation of chaos holds with a rate (Theorem~\ref{thm.main20}).  
\vs
To write the results in this introduction requires considerable notation. Thus we postpone stating the precise theorems to  section~2. 
\vs

We explain, however,  the very general idea of proof. We identify the open and dense in time and space of probabilities set $\mathcal O$ where $\ms U$ is smooth as the set of initial conditions $(t_0,m_0)$  from which starts a unique (in a strong sense) and stable (in a suitable linear sense) minimizer of $J^\infty$.  This step, which is reminiscent of ideas from standard optimal control (see, for instance, the book of Cannarsa and  Sinestrari \cite{CaSiBook}), is similar to what was obtained by Briani and Cardaliaguet  \cite{BrCa} for MFC problems in a different framework, namely,  the state space is the torus and the initial measures have smooth densities.  Here,  the state space is all of  $\R^d$ and the initial conditions are arbitrary probability measures.  To show that $\mathcal U$ is smooth in $\mathcal O$, we adapt ideas used in the construction of a solution to the master equation in mean field games given in Cardaliaguet, Delarue, Lasry and Lions \cite{CDLL} but we argue without the convexity, which translates to monotonicity for general mean field games, assumption used extensively there. 
Then, using  the regularity of $\mathcal U$,  we  derive the propagation of chaos property  for the optimal solution of the particle system when starting from the set $\mathcal O$. 
The key  argument is the fact that the optimal trajectories of  \eqref{takis1} that start in $\mathcal O$ remain there,  while $\mathcal U$ is smooth  and satisfies (almost) the same Hamilton-Jacobi equation as $\mathcal V^N$ in $\mathcal O$.  
\vs 

We finally comment about   possible extensions to problems with common noise { emphasizing once more  that we do not assume any monotonicity/convexity. }The convergence (with algebraic rate)  of $\mathcal V^N$ to $\mathcal U$  for MFC problems with common noise  was established in \cite{CDJS}.
However, the generalization of the results of the present paper to such a setting is far from clear. 
For example, one of the basic tools we use to prove that the set $\mathcal O$ is open and dense is a result of Lions-Malgrange-type, which yields  the uniqueness of the mean field game system characterizing the optimal solution (see system \eqref{MFG} below), but with initial conditions both in $u$ and $m$: see the proof of Lemma \ref{lem.Odense}. 
At this time, we do not know if there is a counterpart of this argument for problems with a common noise, where the mean field system becomes a system of forward-backward stochastic partial differential equations (see \cite{CDLL}). 
\vs

\noindent {\bf Organization of the paper.} We conclude the introduction with the notation we use throughout the paper. In section \ref{sec.1}, we introduce the standing  assumptions,  present the main results and recall some preliminary facts that are needed for the rest of the paper. 
Section \ref{sec.proofThm1} is about  establishing the  smoothness of $\ms U$ (Theorem \ref{thm.mainreguU0}). Section~\ref{sec.proofThm2} is devoted to showing the propagation of chaos property (Theorem~\ref{thm.main20}).

\subsection*{Notations} 
We work on $\R^d$ and write $B_R$ for the open ball centered at the origin with radius $R$ and $I_d$ for the identity matrix.
\vs
We denote by $\mathcal P(\R^d)$ the set of Borel probability measures on $\R^d$. Given $m \in \mathcal P(\R^d)$ and $p \geq 1$, we write $M_p(m)$ for  the $p^{th}-$moment of $m$, that is, $M_p(m) = \int_{\R^d} |x|^p dm$. Then  $\mathcal P_p(\R^d)$ is the set of $m \in \mathcal P(\R^d)$ such that $M_p(\R^d) < \infty$. We endow $\mathcal P_p(\R^d)$ with the Wasserstein metric ${\bf d}_p$, defined by 
\begin{align*}
    {\bf d}^p_p(m, m') : = \inf_{\pi \in \Pi(m,m')} \int_{\R^d} |x - y|^p d\pi(x,y), 
\end{align*}
where $\Pi(m,m')$ is the set of all $\pi \in \mathcal P(\R^d \times \R^d)$ with marginals $m$ and $m'$. For $p=1$, we recall the duality formula 
\begin{align*}
    {\bf d}_1(m,m') = \sup_{\phi} \int_{\R^d} \phi d(m- m'), 
\end{align*}
where the supremum is taken over all 1-Lipschitz maps $\phi : \R^d \to \R$.
\vs 
For $\bx = (x^1,...,x^N) \in (\R^d)^N$, $m_{\bx}^N \in \mathcal P(\R^d)$ is  the empirical measure of $\bx$, that is,  $m_{\bx}^N = \frac{1}{N} \sum_{i = 1}^N \delta_{x^i}$. 
\vs
Given a map $U:\Pk\to \R$, we denote by $\delta U/\delta m$ its flat derivative and by $D_mU$ its Lions-derivative when these derivatives exist, and we use the corresponding notations for second order derivatives. We refer to \cite{CDLL} and  the books of Carmona and Delarue \cite{CaDeBook} for definitions and properties. 
\vs 
We write $C^k_{loc}$ and $C^k$ for the sets of maps on $\R^d$ with continuous and continuous and bounded derivatives up to order $k$. For $r>0$, $r\notin \N$,  we denote by $C^{r,2r}$ the standard parabolic H\"{o}lder spaces 
and by $C^{r,2r}_c$ the subset of functions of $C^{r,2r}$ with a compact support.
\vs
Given a topological vector space,  we write by $E'$ for  its dual space. 
\vs
We often need to  compare  continuous maps defined on different intervals of $[0,T]$. For this,  we simply extend the maps continuously on $[0,T]$ by a constant. For instance, if $t_0\in (0,T]$, $E$ is a topological space and $f:[t_0,T]\to E$ is continuous, we set $f(s)=f(t_0)$ for $s\in [0,t_0]$. 
\vs
{Finally, throughout the proofs $C$ denotes a positive constants, which, unless otherwise said, depends on the data and may change from line to line. }

\subsection*{Acknowlegments} Cardaliaguet was partially supported by the Air Force Office for Scientific Research grant FA9550-18-1-0494 and IMSI, the Institute for Mathematical and Statistical Innovation. Souganidis was partially supported  by the National Science Foundation grant DMS-1900599, the Office for Naval Research grant N000141712095 and the Air Force Office for Scientific Research grant FA9550-18-1-0494. Both  authors would like to thank the IMSI for its hospitality during the Fall 2021 program. 

\section{The assumptions, the main results and some preliminary facts}\label{sec.1}

\subsection*{The standing assumptions} We state our standing assumptions on the maps $H:\R^d\times \R^d\to \R$, $\mathcal F:\Pk\to \R$ and $\mathcal G:\Pk\to \R$, which constitute the data of our problem. We recall that $L:\R^d\times \R^d\to \R$ is the Legendre transform of $H$ with respect to the second variable: 
$$
L(x,a)= \sup_{p\in \R^d} -a\cdot p-H(x,p).
$$

We assume that 
\be\label{convex}
\begin{cases}
\text {$H=H(x,p):\R^d\times \R^d\to \R$ is of class $C^4_{loc}$ and strictly convex with respect to the}\\[1.2mm]
\text{second  variable, that is, for each  $R>0$ there exists}\\[1.2mm]
\text{ $c_R, C_R>0$ such that,   for all $(x,p)\in \R^d\times   \overline{B_R}$,}\\[1.2mm]
\qquad \qquad D^2_{pp}H(x,p)\geq c_R I_d \ \  \text{and} \ \  |D^2_{xx}H(x,p)| + |D^2_{xp}H(x,p)|\leq C_R,
\end{cases}
\ee
 
\be\label{hyp.growth_1}
\begin{cases}
\text{there exists a constant $C>0$ such that, for all $ (x,p)\in \R^d\times \R^d $,}\\[1.2mm] 
 -C+C^{-1}|p|^2\leq H(x,p)\leq C(1+|p|^2)\ \  \text{and} \ \ |D_xH(x,p)| \leq C(|p|+1),
\end{cases}
\ee

\be\label{F}
\text{$\mathcal F:\Pk\to \R$ is of class $C^2$ with $\mathcal F$, $D_m\mathcal F$, $D^2_{ym}\mathcal F$ and $D^2_{mm}\mathcal F$ uniformly bounded,}
\ee
\be\label{G}
\text{$\mathcal G:\Pk\to \R$ is of class $C^4$ with all derivatives up to order $4$ uniformly bounded.}
\ee
\vs

For simplicity, in what follows we put together all the assumptions above in
\be\label{ass.main}
\text{assume that \eqref{convex}, \eqref{hyp.growth_1},  \eqref{F} and \eqref{G} 
hold.}
\ee

The assumptions on the Hamiltonian $H$ are fairly standard, although a little restrictive, and are  used  in \cite{CDJS}  to obtain, independent of $N$, Lipschitz estimates on the value function $\mathcal V^N$; see Lemma \ref{lem.estiVN}. An example  satisfying \eqref{ass.main} is a Hamiltonian of the form $H(x,p)= |p|^2+ V(x)\cdot p$ for some smooth and globally Lipschitz continuous vector field $V:\R^d\to \R^d$. The regularity conditions on $\mathcal F$ and $\mathcal G$ are also important to obtain the estimates of Lemma \ref{lem.estiVN} and to prove the regularity of $\mathcal U$. 

\subsection*{The results}
Given $(t_0, {\bf x})\in [0,T)\times (\R^d)^N$,  $\mathcal V^N(t_0, {\bf x})$ is the value function of the optimal control of the $N-$particle problem given by \eqref{def.VN}.
\vs
We now define in a rigorous way the value function $\mathcal U$ of the MFC, which was informally introduced in \eqref{takis1}. 
For each initial point $(t_0,m_0)\in [0,T)\times\Pk $, we use the set $\ms M(t_0,m_0)$ of controls given by 

\be\label{takis10}
\ms M(t_0,m_0):=\left\{\begin{array}{l}(m,\alpha)\in C^0([t_0,T], \Pk)\times L^0([t_0,T]\times \R^d; \R^d ) \;:  \\[1.5mm]
\int_0^T\int_{\R^d}|\alpha|^2m<\infty \ \ \text{and} \ \  \\[1.5mm]
\partial_t m-\Delta m +{\rm div}(m\alpha) = 0\; \; {\rm in}\; \; (t_0,T)\times \R^d \; \;\text{and} \; \; m(0)=m_0\; {\rm in}\;  \R^d\end{array}\right\},
\ee
with the equation in \eqref{takis10} understood in the sense of distributions.
\vs

Then the value function $\mathcal U(t_0,m_0)$ of the MFC  problem is given by 
\be\label{takis19}
\mathcal U(t_0,m_0):=\underset{(m,\alpha)\in \ms M(t_0,m_0)} \inf\left\{ \int_{t_0}^T ((\int_{\R^d} L(x, \alpha(t,x))m(t,dx))+\mathcal F(m(t)) dt +\mathcal G(m(T))\right\}.
\ee

Our first result is about the regularity of $\mathcal U$ in the set $\mathcal O$ defined by 
\be\label{takis20}
\mathcal O:=\left\{\begin{array}{l} (t_0,m_0)\in [0,T)\times \Pw \: : \mbox{\rm there exists a unique stable minimizer}\\[1.5mm]
\text{ in the definition of $\mathcal U(t_0,m_0)$}
\end{array}\right\}.
\ee
The notion of stability is defined in terms of the linearized MFG system and is introduced in section~\ref{sec.proofThm1}. 
\vs

The first main result is stated next.

\begin{thm}\label{thm.mainreguU0} Assume \eqref{ass.main}. The value function  $\mathcal U$ is globally Lipschitz continuous on $[0,T]\times \Pk$ and of class $C^1$ in the set  $\mathcal O$, which is open and dense in $[0,T]\times \Pw$. Moreover,  $\mathcal U$ is  a classical solution  in $\mathcal O$ of  the master Hamilton-Jacobi equation
\be\label{HJdsO0}
\begin{array}{l}
\ds -\partial_t \mathcal U(t,m)-\int_{\R^d}{\rm div}( D_m\mathcal U(t,m,y))m(dy) +\int_{\R^d} H(y, D_m\mathcal U(t,m,y))m(dy)=\mathcal F(m).
\end{array}
\ee
In addition, for any $(t_0,m_0)$ in $\mathcal O$, there exists $\ep>0$ and a constant $C>0$, which depends on $(t_0,m_0)$ and is such that, for any $t\in [0,T]$, $x,y \in \R^d$ and $m^1,m^2\in \Pw$ with  $|t-t_0|<\ep$, ${\bf d_2}(m_0,m^1)<\ep$ and ${\bf d_2}(m_0,m^2)<\ep$,  
\be\label{takis50}
\left| D_m\mathcal U(t,m^1,x)-D_m\mathcal U(t,m^2,y)\right| \leq C(|x-y|+ {\bf d}_2(m^1,m^2)).
\ee
\end{thm}
\vs
By a classical solution of \eqref{HJdsO0}, we mean that the derivatives of $\mathcal U$ involved in the equation exist and are continuous.
\vs

Our second main result  is a quantitative propagation of chaos property  about the  optimal trajectories of the underlying  $N-$particle system. 

\begin{thm} \label{thm.main20} Assume \eqref{ass.main}. There exists a constant $\gamma\in (0,1)$ depending only on the  dimension $d$ such that, for every  $(t_0,m_0)\in \mathcal O$ with $M_{d+5}(m_0)<+\infty$, there is $C=C((t_0,m_0))>0$ 
such that, if ${\bf Z}=(Z^k)_{k=1, \dots, N}$ is a sequence of independent random variables with law $m_0$, ${\bf B_.}=(B^k_.)_{k=1, \dots, N}$ is a sequence of independent Brownian motions independent of ${\bf Z}$, %the $(Z^k)$ 
and ${\bf Y_.}^N=(Y^{N,k})_{k=1, \dots, N}$ is the optimal trajectory for $\mathcal V^N(t_0, (Z^k)_{k=1, \dots, N})$, that is, for each $k=1,\ldots,N$ and $t\in [t_0,T]$, 
\be\label{takis51}
Y^{N,k}_t= Z^k-\int_{t_0}^tH_p(Y^k_s, D\mathcal V^N(s, {\bf Y}^N_s))ds +\sqrt{2}(B^k_t-B^k_{t_0}) , 
\ee
then 
$$
\E\left[ \sup_{t\in [t_0,T]}{\bf d_1}(m^N_{{\bf Y}^N_t}, m(t))\right] \leq CN^{-\gamma}.
$$
\end{thm}

\subsection*{Some preliminary facts}\label{subsec.preliminaries} We recall here some well known facts about MFC that we use in the paper.  
\vs
We begin with some regularity properties of the underlying backward-forward MFG system.  
Fix  $(t_0, m_0)\in [0,T]\times \Pw$. We recall (see, for example,  \cite{LL06cr2, LLJapan} for the original statement or \cite{daudin21}, in a framework closer to our setting) that there exists at least one  minimizer for $\mathcal U(t_0,m_0)$ and that, if $(m, \alpha)\in \ms M(t_0,m_0)$ is a minimizer, then there exists a multiplier $u:[t_0,T]\times \R^d\to \R$ such that $\alpha= -H_p(x,Du)$ and  the pair $(u,m)$ solves the MFG-system 
\be\label{MFG}
\left\{\begin{array}{l}
\ds -\partial_t u -\Delta u +H(x,Du)= F(x,m(t)) \ \  {\rm in} \ \ (t_0,T)\times \R^d,\\[1.2mm]
\ds \partial_t m -\Delta m -{\rm div}(H_p(x,Du)m)=0 \ \  {\rm in} \ \ (t_0,T)\times \R^d, \\[1.2mm]
\ds m(t_0)=m_0, \; u(T, x)=G(x,m(T)) \ \  {\rm in} \ \  \R^d,
\end{array}\right.
\ee
where 
$$
F(x,m)=\frac{\delta \mathcal F}{\delta m}(m,x) \ \ \text{and} \ \ G(x,m)=\frac{\delta \mathcal G}{\delta m}(m,x).
$$
Note that, in view of the assumed strict convexity of $H$, given $(m,\alpha)$,   the relation $\alpha= -H_p(x,Du)$ defines uniquely $Du$.

\begin{lem} \label{lem.reguum} Assume \eqref{ass.main} and let $(u,m)$ be a the solution of \eqref{MFG}.  Then, for any $\delta\in (0,1)$,  there exists $C>0$, which is independent of $(t_0,m_0)$, such that 
\be\label{reguum}
\|u\|_{C^{2+\delta/2,4+\delta}} +\sup_{t\neq t'} \frac{{\bf d_2}(m(t), m(t'))}{|t-t'|^{1/2}} \leq C\ \ \  {\rm and }\ \ \ \sup_{t\in [t_0,T]} \int_{\R^d}|x|^2 m(t,dx) \leq C
\int_{\R^d}|x|^2 m_0(dx).
\ee
\end{lem}

\begin{rmk} Note that, under our standing assumptions, we do not have, in general,  uniqueness of the solution to \eqref{MFG}.  Indeed, the problem defining $\mathcal U(t_0,m_0)$ may have several minimizers and/or  solutions of  \eqref{MFG} may not necessarily be associated with a minimizer of $\mathcal U(t_0,m_0)$. 
\end{rmk}

\begin{proof} The estimates of $m$ and  the local regularity of $u$ are standard.  Indeed, the uniform bound on $Du$  follows as in the proof of Lemma 3.3 in \cite{CDJS}, and then  the estimate on $m$ are immediate. Moreover,  the local regularity of $u$ is a consequence of the classical parabolic regularity theory. 
\vs
The only point is to explain why this regularity holds globally in space. For this, we first note that the assumptions  on $\mathcal F$ and $\mathcal G$ yield a $C>0$ such that, for all $x,y\in \R^d$, 
$$
\sup_{t\in [t_0,T]} \left| F(x,m(t))-F(y,m(t))\right| + \left| G(x,m(t))-G(y,m(t)\right|\leq C|x-y|.
$$
It then follows from the maximum principle that $u$ is, uniformly in $(t_0,m_0)$ and in $t$, Lipschitz continuous in the space variable. The same argument applied to the equation satisfied by $u_{x_i}$ for each $i=1,\ldots,d$, implies that 
$Du$ is also uniformly Lipschitz continuous in the space variable. 
\vs
The general conclusion can be established similarly, using the maximum principle for the global estimates and the parabolic regularity for the local one. 
\end{proof}

In view  of the uniform estimates in \eqref{reguum}, we have the following stability for minimizers in \eqref{takis19} when they are unique.

\begin{lem}\label{lem.stabilo} Assume \eqref{ass.main}, fix  $(t_0,m_0)\in [0,T)\times \Pw$, suppose that $\mathcal U(t_0,m_0)$ has a unique minimizer $(m, \alpha)$, and let $u$ be the associated multiplier. If $(t_0^n, m_0^n)$ converges to $(t_0,m_0)$ and if $(m^n,\alpha^n)$ is a minimizer for $\mathcal U(t_0^n,m_0^n)$ with associated multiplier $u^n$, then $u^n$, $Du^n$ and $D^2u^n$ converge respectively to $u$, $Du$ and $D^2u$  in $C^{\delta/2,\delta}$. In addition, if $t_0^n=t_0$ for all $n$, the convergence of $(u^n)$ holds in $C^{(2+\delta)/2,2+\delta}$.
\end{lem} 

\begin{proof} It easily follows from the regularity of $u$  (see \eqref{reguum}) that, without loss of generality, we can assume that $t^n_0=t_0$. Moreover, again in view of \eqref{reguum} and the continuity of $\mathcal U$ (see Lemma~\ref{lem.estiVN} below), %which states that $\mathcal U$ is continuous, 
the minimizer $(m^n,\alpha^n)$ converge along subsequences in $C^0([t_0,T]\times \Pk)\times C^{1,2}_{loc}([t_0,T]\times \R^d)$ to  minimizers for $\mathcal U(t_0,m_0)$. 
% in $C^0([t_0,T]\times \Pk)\times C^{1,2}_{loc}([t_0,T]\times \R^d)$. 
Since the latter is assumed to have a unique minimizer $(m,\alpha)$, the convergence holds along the whole sequence. Arguing as in Lemma \ref{lem.reguum}, we   can also check that the convergence of the $u^n$'s holds in $C^{(2+\delta)/2,2+\delta}$, because $z^n=u^n-u$ solves a linearized  equation of the form 
$$
 -\partial_t z^n-\Delta z^n+V^n\cdot Dz^n= F(x,m^n(t))-F(x,m(t)) \ \ \ z^n(T,x)= G(x,m^n(T))-G(x,m(T)),
 $$
 with  $V^n(t,x)= \int_0^1 H_p((1-s)Du^n+sDu,x)ds$,  and where, in view of the regularity of $\mathcal F$ and $\mathcal G$  and the convergence of $m^n$ to $m$,  $F(\cdot,m^n(\cdot))-F(\cdot,m(\cdot))$ and $G(\cdot,m^n(T))-G(\cdot,m(T))$ converge as $n\to \infty$ to $0$.
\end{proof}

The following second-order optimality condition is used several times in the proofs of the main results.

\begin{lem}\label{lem.secondordercond} Assume \eqref{ass.main}, fix  $(t_0,m_0)\in [0,T)\times \Pw$ and let $(m, \alpha)$ be a minimizer for $\mathcal U(t_0,m_0)$. Fix  $\beta\in C^0([t_0,T]\times \R^d;\R^d)$ or $\beta \in L^\infty([t_0,T]\times \R^d; \R^d)$ with $\beta=0$ in a neighborhood of $t_0$ and let  $\rho\in C^0([t_0,T], (C^{2+\delta}(\R^d))')$ be the solution in the sense of distributions to
\be\label{eq.forsecondorder}
\left\{\begin{array}{l}
\partial_t \rho-\Delta \rho +{\rm div} (\rho \alpha)+ {\rm div}(m\beta)= 0 \ \  {\rm in} \ \  (t_0,T)\times \R^d,\\[1.2mm]
\rho(t_0)= 0 \ \  {\rm in} \ \  \R^d.
\end{array}\right. 
\ee
Then
\be\label{ineq.SOC}
\begin{split}
& \int_{t_0}^T \Bigl( \int_{\R^d} L_{\alpha,\alpha}(x,\alpha(t,x)) \beta(t,x)\cdot   \beta(t,x)m(t,dx) +\lg  \frac{\delta F}{\delta m}(\cdot,m(t),\cdot), \rho(t)\otimes \rho(t)\rg\Bigr)dt\\
& \qquad + \lg \frac{\delta G}{\delta m}(\cdot,m(T),\cdot), \rho(T,\cdot)\otimes \rho(T,\cdot)\rg\geq 0.
\end{split}
\ee
\end{lem}
This statement is an adaptation of an analogous result in \cite{BrCa}. The existence of the solution to \eqref{eq.forsecondorder} and the proof of \eqref{ineq.SOC} are given in the Appendix. \\

It is well-known that the map $\mathcal V^N$ defined in \eqref{def.VN} solves the uniformly parabolic Hamilton-Jacobi-Bellman (HJB for short) equation 
$$
\left\{\begin{array}{l}
\ds -\partial_t \mathcal V^N(t,{\bf x}) -\sum_{j=1}^N \Delta_{x^j}\mathcal V^N(t,{\bf x}) +\frac1N\sum_{j=1}^N H(x^j, N D_{x^j}\mathcal V^N(t, {\bf x}))=\mathcal F(m^N_{{\bf x}})\ \ {\rm in}\ \  (0,T)\times (\R^d)^N \\[1.5mm]
\ds  \mathcal V^N(T,{\bf x})= \mathcal G(m^N_{{\bf x}})\ \ {\rm in}\ \ (\R^d)^N,
\end{array}\right.
$$
and, therefore $\mathcal V^N$ is smooth for any $N$. This is in contrast with the limit $\mathcal U$, which might not be $C^1$. The following result, proved in \cite{CDJS}, states however that both maps are uniformly Lipschitz continuous. 

\begin{lem}[Regularity of $\mathcal V^N$ and of $\mathcal U$] \label{lem.estiVN}  Assume \eqref{ass.main}. There exists constant $C$ depending on the data such that 
$$
\|\mathcal V^N\|_\infty+N \sup_{j=1,\ldots,N} \|D_{x^j}\mathcal V^N\|_\infty \leq C,
$$
and, for all  $(t,m), (t',m')\in [0,T]\times \Pk,$ 
$$
\left|\mathcal U(t,m)-\mathcal U(t',m')\right| \leq C(|t-t'|+{\bf d}_1(m,m')). 
$$
\end{lem}
\vs

The following convergence rate is the main result of \cite{CDJS}. 

\begin{prop}[Quantified convergence of $\mathcal V^N$ to $\mathcal U$] \label{prop.quantifCV} Assume \eqref{ass.main}. There exists $\gamma\in (0,1)$ depending on dimension only and $C>0$ depending on the smoothness of the data such that, for any $(t,\bx)\in [0,T]\times (\R^d)^N$, 
$$
\left| \mathcal V^N(t,\bx)-\mathcal U(t,m^N_{\bx})\right|  \leq C \dfrac{1}{N^{\gamma}} \left( 1+ M_2(m^N_{{\bf x}})\right).
$$
\end{prop}

\section{The regularity of $\mathcal U$}\label{sec.proofThm1}

We prove here Theorem \ref{thm.mainreguU0}. A crucial  step is the analysis of a linearized system, which is reminiscent of a linearized system studied in \cite{CDLL} and \cite{BrCa} for MFG problems.  The main and important difference from \cite{CDLL} is that here we do not assume that $\ms F$ and $\ms G$ are convex, while 
 \cite{BrCa} deals with problems on the torus. We go around the lack of monotonicity by using the notions of 
stability and strong stability of a solution, which are introduced next using the linearized system. Finally, stability  is also  used  to define and analyze the open, dense set $\mathcal O$ on which the map $\mathcal U$ will eventually be smooth. 

\subsection*{The linearized system}  We fix $t_0\in [0,T)$, a constant $C_0$,  and, for  $m_0\in \Pw$ and $V:[t_0,T]\times \R^d\to \R^d \ \ \text{with} \ \  \|V\|_{C^{1,3}}\leq C_0, $ let  $m$  be  the solution to 
\be\label{eq.m-Appen}
\left\{\begin{array}{l}
\partial_t m -\Delta m -{\rm div}(Vm)=0\ \  {\rm in}\ \ (t_0,T)\times \R^d  \\[1.2mm]
m(t_0)=m_0\ \  {\rm in}\ \  \R^d.
\end{array}\right.
\ee

We analyze the inhomogeneous linearized system 
\be\label{eq.LS}
\left\{ \begin{array}{l}
\ds -\partial_t z -\Delta z +V(t,x)\cdot Dz = \frac{\delta F}{\delta m}(x,m(t))(\rho(t))+R^1(t,x)\ \  {\rm in}\ \  (t_0,T)\times \R^d, \\[3mm]
\ds \partial_t\rho-\Delta \rho -{\rm div}(V\rho) -\sigma {\rm div}(m\Gamma Dz) = \sigma {\rm div}(R^2)\ \ {\rm in}\ \  (t_0,T)\times \R^d,\\[1.2mm]
\ds \rho(t_0)=\sigma \xi \ \ \text{and} \ \ z(T,x)= \frac{\delta G}{\delta m}(x,m(T))(\rho(T))+R^3\ \  {\rm in}\ \   \R^d,
\end{array}\right.
\ee
where  
\be\label{takis30}
\begin{split}
& \sigma\in [0,1] \ \ \text{and} \ \ \delta\in (0,1),\\[1.2mm]
& \Gamma\in C^0([0,T]\times \R^d; \R^{d\times d}) \ \ \text{ with}  \ \ \|\Gamma\|_\infty\leq C_0, \\[1.2mm]
& R^1\in C^{\delta/2,\delta}, \ \  R^2\in L^\infty([t_0,T], (W^{1,\infty})'(\R^d,\R^d)), \ \  R^3\in C^{2+\delta} \ \ \text{ and} \ \  \xi\in (W^{1, \infty})'. 
\end{split}
\ee

The pair  $(z,\rho) \in C^{0,1}([t_0,T]\times\R^d) \times C^0([0,T], (C^{2+\delta})')$ is a solution to \eqref{eq.LS} if $z$
and $\rho$ satisfy respectively  the first and second equation in the sense of distributions.
\vs
Note that, because of the regularity of $\rho$ and the assumptions on $\mathcal F$ and $\mathcal G$, the maps $(t,x)\to  \frac{\delta F}{\delta m}(x,m(t))(\rho(t))$ and  $x\to \frac{\delta G}{\delta m}(x,m(T))(\rho(T))$ are continuous and bounded. \

\vs

We will often use system \eqref{eq.LS} in which $V(t,x)= H_p(x,Du(t,x))$ and $\Gamma(t,x)=H_{pp}(x,Du(t,x))$, where $(u,m)$ is a classical solution to \eqref{MFG}. In this case, $V$, $\Gamma$ and $m$  satisfy the conditions above. 
\vs

Next we introduce the notion of strong stability for the homogeneous version of \eqref{eq.LS}, that is the system
\be\label{eq.LSH}
\left\{ \begin{array}{l}
-\partial_t  z -\Delta  z +V(t,x)\cdot D z = \frac{\delta F}{\delta m}(x,m(t))( \rho(t)) \ \  {\rm in}\ \ (t_0,T)\times \R^d, \\[1.2mm]
\partial_t \rho-\Delta  \rho -{\rm div}(V \rho) -\sigma{\rm div}(m\Gamma D z) =0\ \  {\rm in}\ \ (t_0,T)\times \R^d, \\[1.2mm]
\rho(t_0)= 0 \ \ \text{and} \ \  z(T,x)= \frac{\delta G}{\delta m}(x,m(T))( \rho(T))\ \  {\rm in}\ \  \R^d. 
\end{array}\right.
\ee

We say that 
\be\label{takis31}
\text{the system \eqref{eq.LSH} is strongly stable if, for  any $\sigma\in [0,1]$, its  unique solution is $(z,\rho)=(0,0)$.}
\ee

The main result of the subsection is  a uniqueness and regularity result for the solution to  \eqref{eq.LS}. 

\begin{lem} \label{lem.estLS-Appen} Assume \eqref{ass.main} and \eqref{takis31}. 
 There exist a neighborhood $\mathcal V$ of $(V,\Gamma)$ in the topology of locally uniform convergence, and $\eta, C>0$ such that, for any $(V', t_0', m_0', \Gamma', R^{1,'}, R^{2,'}, R^{3,'}, \xi',\sigma')$ with 
\be\label{cond.data-Appen}
\begin{split}
 (V',\Gamma')\in \mathcal V,\quad |t_0'-t_0|+{\bf d}_2(m_0',m_0) \leq \eta, \ \  \|V'\|_{C^{1, 3}}+\|\Gamma'\|_\infty \leq 2 C_0, \ \  \sigma'\in [0,1],\\[1.2mm]
 R^{1,'}\in C^{\delta/2,\delta}, \; R^{2,'}\in L^\infty([t_0,T], (W^{1,\infty})'(\R^d,\R^d)), \; R^{3,'}\in C^{2+\delta},
 \; \xi'\in (W^{1, \infty})',
 \end{split}
 \ee
any solution $(z',\rho')$ to \eqref{eq.LS} associated with these data on $[t_0',T]$ and $m'$ the solution to \eqref{eq.m-Appen} with drift $V'$ and initial condition $m_0'$ at time $t_0'$ satisfies 
\be\label{takis32}
\|z'\|_{C^{(2+\delta)/2,2+\delta}}+ \sup_{t\in [t_0',T]} \|\rho'(t,\cdot)\|_{(C^{2+\delta})'} + \sup_{t'\neq t} \frac{\|\rho'(t',\cdot)-\rho'(t,\cdot)\|_{(C^{2+\delta})'}}{|t'-t|^{1/2}} \leq CM' , 
\ee
where 
\be\label{def.Mprime-Appen}
M':= \|\xi'\|_{(W^{1,\infty})'}+\|R^{1,'}\|_{C^{\delta/2, \delta}}+\|R^{3,'}\|_{C^{2+\delta}}+ \sup_{t\in [t_0',T]} \|R^{2,'}(t)\|_{(W^{1,\infty})'}.
\ee
\end{lem}

An immediate consequence is the following corollary.

\begin{cor}
Assume \eqref{ass.main} and \eqref{takis31}. Then, for any $(V', m_0', \Gamma')$ satisfying \eqref{cond.data-Appen},
 the corresponding homogeneous linearized system is strongly stable.
\end{cor}

The proof of Lemma~\ref{lem.estLS-Appen} follows  some of the  ideas of \cite{CDLL}, where a similar system is studied. The main differences are  that, here, we use the stability condition instead of the monotonicity assumption of \cite{CDLL} and  work in an unbounded space. 
\vs

In what follows, we need a preliminary result which we state and prove next. The difference between the estimate below and the one of Lemma~\ref{lem.estLS-Appen} is the right hand side of the former which depends on the solution itself.

\begin{lem}\label{lem.apriori-Appen} Assume \eqref{ass.main} and let $(z,\rho)$ be a solution to \eqref{eq.LS}. There is a constant $C$, depending only on the regularity of $\mathcal F$, $\mathcal G$ and on $\|V\|_{C^{1,3}}+ \|\Gamma\|_\infty$, such that
$$
\|z\|_{C^{(2+\delta)/2,2+\delta}}+ \sup_{t\in [t_0,T]} \sup_{t'\neq t} \frac{\|\rho(t',\cdot)-\rho(t,\cdot)\|_{(C^{2+\delta})'}}{|t'-t|^{\delta/2}} \leq C\left(M+\sup_{t\in [t_0,T]} \|\rho(t)\|_{(C^{2+\delta})'}\right)
$$
where 
$$
M:=\|\xi\|_{(W^{1,\infty})'}+ \|R^{1}\|_{C^{\delta/2, \delta}} + \sup_{t\in [t_0,T]}\|R^{2}(t)\|_{(W^{1,\infty})'}+ \|R^{3}\|_{C^{2+\delta}}.
$$
\end{lem}

\begin{proof}  Throughout the proof, $C$ denotes a constant that depends only on the data and may change from line to line. 
\vs

Set $R= \sup_{t\in [t_0,T]} \|\rho(t)\|_{(C^{2+\delta})'}<\infty$. 
It follows that the maps $(t,x)\to  \frac{\delta F}{\delta m}(x,m(t))(\rho(t))$  and $x\to  \frac{\delta G}{\delta m}(x,m(T))(\rho(T))$  are  bounded by $CR$, the latter  in    
$C^{2+\delta}$.
\vs

Then, standard parabolic regularity gives that  $z$ is  bounded in $C^{(1+\delta)/2, 1+\delta}$ by $CR$. 
\vs

The main step of the proof is to show that 
$$
\sup_{t\in [t_0,T]} \sup_{t'\neq t} \frac{\|\rho(t')-\rho(t)\|_{(C^{2+\delta})'}}{|t'-t|^{\delta/2}} \leq C(M+R).
 $$
Arguing by duality, we fix $t_0\leq t_1<t_2\leq T$ and, with $\bar \psi\in C^{2+\delta}$,  we consider the solution  $\psi^i$ for $i=1,2$   to 
 \be\label{eq.dualeqTER}
\left\{\begin{array}{l}
-\partial_t \psi^i -\Delta \psi^i + V(t,x) \cdot D\psi^i = 0\ \  {\rm on}\ \  (t_0,t_i)\times \R^d,\\[1.5mm]
\psi(t_i)= \bar \psi \ \ {\rm in}\ \ \R^d,
\end{array}\right.
\ee
which, in view of the assumption on $V$ and parabolic regularity, satisfies, for $i=1,2$, the bound 
\[\|\psi^i\|_{C^{(2+\delta)/2,2+\delta}}\leq C\|\bar \psi\|_{C^{2+\delta}}.\]
 
It is immediate  that, for $i=1, 2$,  
$$
\lg \rho(t_i), \bar \psi\rg = \lg \xi, \psi^i(t_0)\rg-\sigma \int_{t_0}^{t_i} \int_{\R^d} \Gamma Dz\cdot D\psi^i m dx dt-\sigma \int_{t_0}^{t_i}\lg R^{2}(t),  D\psi^i(t)\rg dt. 
$$ 
Thus
\be\label{takis35}
\begin{split}
& \lg \rho(t_2)-\rho(t_1), \bar \psi\rg  = \lg \xi, \psi^2(t_0)-\psi^1(t_0)\rg-\sigma \int_{t_0}^{t_1} \int_{\R^d} \Gamma Dz\cdot (D\psi^2-D\psi^1) m \\[3mm]
&  - \sigma \int_{t_1}^{t_2} \int_{\R^d} \Gamma Dz \cdot D\psi^2 m  -\sigma \int_{t_0}^{t_1} \lg R^{2}(t),  (D\psi^2-D\psi^1)(t) \rg dt - \sigma\int_{t_1}^{t_2} \lg R^{2}(t),  D\psi^2(t)\rg  dt. 
\end{split}
\ee
Note that $\psi^2-\psi^1$ solves \eqref{eq.dualeqTER} on $(t_0,t_1)$ with a terminal condition at $t_1$ given by $(\psi^2-\psi^2)( t_1,\cdot)=\psi^2(t_1,\cdot)-\psi^2 (t_2,\cdot)$,  which,  in view of the regularity of $\psi^2$,  is bounded  in $W^{2,\infty}$ by $C(t_2-t_1)^{\delta/2}\|\bar \psi\|_{C^{2+\delta}}$.  
\vs
It then follows follows from the maximum principle that   $(\psi^2-\psi^1)( t,\cdot)$ is bounded in $W^{2,\infty}$ by $C(t_2-t_1)^{\delta/2}\|\bar \psi\|_{C^{2+\delta}}$ for any $t$, and,  hence, 
\begin{align*}
&\left|\lg \xi, \psi^2( t_0,\cdot)-\psi^1( t_0,\cdot)\rg\right|   \leq  \|\xi\|_{(W^{1, \infty})'}\|\psi^2( t_0,\cdot)-\psi^1( t_0,\cdot)\|_{W^{1,\infty}} \\[2mm]
& \leq C(t_2-t_1)^{\delta/2}\|\bar \psi\|_{C^{2+\delta}}\|\xi\|_{(W^{1,\infty})'} \leq C(t_2-t_1)^{\delta/2}\|\bar \psi\|_{C^{2+\delta}}M,
\end{align*}
$$
\left|  \int_{t_0}^{t_1} \int_{\R^d} \Gamma Dz \cdot (D\psi^2-D\psi^1) m dx dt\right| \leq C\|Dz\|_\infty(t_2-t_1)^{\delta/2}\|\bar \psi\|_{C^{2+\delta}}\leq  C(t_2-t_1)^{\delta/2}\|\bar \psi\|_{C^{2+\delta}}R
$$
and 
\begin{align*}
\left|\int_{t_0}^{t_1} \lg R^{2}(t), (D\psi^2-D\psi^1)(t) \rg dt\right|& \leq \sup_t \|R^2(t)\|_{(W^{1,\infty})'}\sup_t \|D\psi^2(t)-D\psi^1(t)\|_{W^{1,\infty}}\\
&   \leq C(t_2-t_1)^{\delta/2}\|\bar \psi\|_{C^{2+\delta}}M.
\end{align*}
Note that the regularity of $\psi^2$ yields 
$$
\left| \sigma \int_{t_1}^{t_2} \int_{\R^d} \Gamma Dz \cdot D\psi^2 m dx dt \right| \leq C(t_2-t_1) \|Dz\|_\infty\|D\psi^2\|_\infty \leq C(t_2-t_1)^{1/2}  \|\bar\psi\|_{C^{2+\delta}}R
$$
and 

$$
\left|\int_{t_1}^{t_2} \lg R^{2}(t), D\psi^2(t) \rg dt\right|\leq (t_2-t_1)\sup_{t} \|R^{2}(t)\|_{(W^{1,\infty})'} \sup_{t}\|D\psi^2(t)\|_{W^{1,\infty}} \leq C(t_2-t_1) \|\bar \psi\|_{C^{2+\delta}}M.  
$$
\vs
Since $\bar \psi$ is arbitrary, it follows from \eqref{takis35} that 
$$
\|\rho(t_2)-\rho(t_1)\|_{(C^{2+\delta})'} \leq C(t_2-t_1)^{\delta/2}(M+R). 
$$
This regularity of $\rho$ implies that the maps $(t,x)\to  \frac{\delta F}{\delta m}(x,m(t))(\rho(t))$ and  $x\to \frac{\delta G}{\delta m}(x,m(T))(\rho(T))$ are bounded  in $C^{\delta/2,\delta}$ and $C^{2+\delta}$ respectively by $C(M+R)$. Thus $z$ is also bounded in  $C^{(2+\delta)/2,2+\delta}$ by $C(M+R)$. 
\end{proof}

\begin{proof}[Proof of Lemma \ref{lem.estLS-Appen}] The first (and main) part of the proof consists in showing the existence of $\mathcal V$, $\eta>0$ and $C>0$ such that, for any $(V', t_0', m_0', \Gamma', R^{1,'}, R^{2,'}, R^{3,'}, \xi',\sigma')$ as in \eqref{cond.data-Appen}, any solution $(z',\rho')$ to \eqref{eq.LS} associated with these data on $[t_0',T]$ satisfies, with $M'$ is defined by \eqref{def.Mprime-Appen}, 
$$
\sup_{t\in [t_0',T]} \|\rho'(t)\|_{(C^{2+\delta})'} \leq CM'.
$$
%where $M'$ is defined by \eqref{def.Mprime-Appen}. 
\vs

We prove this claim by contradiction, assuming the existence of  sequences $(V^n)_{n\in \N}$, $(t_0^n, m_0^n)_{n\in \N}$, $(m^n)_{n\in \N}$,   $(\Gamma^n)_{n\in \N}$, $(R^{i,n})_{n\in \N}$, $(\xi^n)_{n\in \N}$,  and $(\sigma^n)_{n\in \N}$ such that $(V^n,\Gamma^n)$ converges locally uniformly to $(V, \Gamma)$,  
$$
|t_0^n-t_0|+{\bf d}_2(m_0^n,m_0)\leq 1/n, \ \  \|V^n\|_{C^{1, 3}}+\|\Gamma^n\|_\infty \leq 2 C_0, \ \  \sigma^n\in [0,1],
$$
and 
$$
\|\xi^n\|_{(W^{1,\infty})'}+\|R^{3,n}\|_{C^{2+\delta}}+\|R^{1,n}\|_{C^{\delta/2,\delta}} + \sup_{t\in [t_0,T]} \|R^{2,n}(t)\|_{(W^{1,\infty})'}  \leq 1,
$$
and, for each $n$, a solution $(z^n,\rho^n)$ to \eqref{eq.LS} associated with the data above, such that 
$$
\lambda^n=\sup_{t\in [t_0^n,T]} \|\rho^n(t)\|_{(C^{2+\delta})'}\geq n.
$$ 
It follows that  $(\tilde z^n, \tilde \rho^n)= \frac{1}{\lambda^n}(z^n, \rho^n)$ solves the system 
\be\label{pierre}
\left\{ \begin{array}{l}
-\partial_t \tilde z^n -\Delta \tilde z^n +V^n(t,x)\cdot D\tilde z^n = \dfrac{\delta F}{\delta m}(x,m^n(t))(\tilde \rho^n(t))+\dfrac{1}{\lambda^n}R_1^n(t,x), \\[2.5mm]
\partial_t\tilde \rho^n-\Delta \tilde \rho^n -{\rm div}(V^n\tilde \rho^n) -\sigma^n{\rm div}(m\Gamma^n D\tilde z^n) = \dfrac{1}{\lambda^n}\sigma^n {\rm div}(R_2^n),\\[2.5mm]
\tilde \rho^n(t_0^n)= \dfrac{1}{\lambda^n}\sigma^n\xi^n \ \ \text{and} \ \  \tilde z^n(T,x)= \dfrac{\delta G}{\delta m}(x,m^n(T))(\tilde \rho^n(T))+\dfrac{1}{\lambda^n}R_3^n.
\end{array}\right.
\ee

Since, by definition, $\sup_t \|\tilde \rho^n(t)\|_{(C^{2+\delta})'}= 1$, Lemma \ref{lem.apriori-Appen} implies that the $(\tilde z^n)$'s and $\tilde \rho^n(t,\cdot)$'s are bounded in $C^{(2+\delta)/2,2+\delta}$ and $C^{\delta/2}([0,T], (C^{2+\delta})')$ respectively. 
\vs

Hence,  we may  assume that,  up to a subsequence, the sequences  $(\sigma^n)_{n\in \N}$, $(\tilde z^n)_{n\in \N}$, $(\tilde \rho^n)_{n\in \N}$ and $(V^n)_{n\in \N}$ converge respectively  to some $\sigma\in [0,1]$,  $\tilde z \in C^{1,2}_{loc}$, $\tilde \rho \in C^0([t_0,T], (C^{2+\delta}_c)')$, where $(C^{2+\delta}_c)'$ is endowed with the weak-$\ast$ topology, and  $V \in C^0([t_0,T], C^{2+\delta}_{loc})$. 
\vs

The goal  is to show that $(\tilde z,\tilde \rho)$ is a nonzero solution to the homogenous equation \eqref{eq.LSH}, which will contradict the strong stability of the system. 
\vs
There are two difficulties that  need to be addressed both caused by the above claimed  weak convergence of the $\tilde \rho^n$'s to $\tilde \rho$.
\vs

The first   is to prove that, as $n\to \infty$,
$$\dfrac{\delta F}{\delta m}(x,m^n(t))(\tilde \rho^n(t)) \to \dfrac{\delta F}{\delta m}(x,m(t))(\tilde \rho(t)) \;\text{and} \; \dfrac{\delta G}{\delta m}(x,m^n(T))(\tilde \rho^n(T)) \to \dfrac{\delta G}{\delta m}(x,m(T))(\tilde \rho(T),$$ 
%is and $\dfrac{\delta G}{\delta m}(x,m^n(T))(\tilde \rho^n(T))$ converge to $\dfrac{\delta F}{\delta m}(x,m(t))(\tilde \rho(t))$ and  $\dfrac{\delta G}{\delta m}(x,m(T))(\tilde \rho(T))$ respectively, 
%which is not ensured by the above convergence of the $\tilde \rho^n$'s. 
\vs
and the second  is to show that, since $\sup_t \|\tilde \rho^n(t)\|_{(C^{2+\delta})'}=1$, we must have    $(\tilde z,\tilde \rho)$ is nonzero.  
%Again this does not follow from the convergence of $(\tilde \rho^n)$'s and requires some additional  argument. 

\vs
To overcome these two issues it is necessary to upgrade the convergence of the $\tilde \rho^n$'s  to  $\tilde \rho$ in \\$C^0([t_0,T], (C^{2+\delta})')$ from weak to strong. 
%actually converge strongly to $\tilde \rho$ in $C^0([t_0,T], (C^{2+\delta})')$. 
\vs
We first note that $\tilde \rho^n=\tilde \rho^{n,1}+\tilde \rho^{n,2}$ with  $\tilde \rho^{n,1}$ and $\tilde \rho^{n,2}$ solving respectively 
$$
\partial_t\tilde \rho^{n,1}-\Delta \tilde  \rho^{n,1} -{\rm div}(V^n\tilde  \rho^{n,1})  = \frac{1}{\lambda^n}\sigma^n {\rm div}(R_2^n)  \ \ \text{and} \ \ 
 \rho^{n,1}(t_0)= \frac{1}{\lambda^n}\sigma^n\xi^n,
$$
and 
\be\label{takis38}
\partial_t\tilde  \rho^{n,2}-\Delta \tilde  \rho^{n,2} -{\rm div}(V^n\tilde  \rho^{n,2})  -\sigma^n{\rm div}(\Gamma^n D\tilde z^nm^n) = 0 \ \ \text{and} \ \ 
\tilde  \rho^{n,2}(t_0)= 0.
 \ee
 \vs
We show next that  $\sup_t \|\tilde \rho^{n,1}(t)\|_{(C^{2+\delta})'} \to 0$. Indeed, for fixed  $\bar t\in (t_0^n,T]$ and $\bar \psi\in C^{2+\delta}$, let $\psi^n$ be the solution to the dual problem
\be\label{eq.dualeqBIS}
\left\{\begin{array}{l}
-\partial_t \psi^n -\Delta \psi^n + V^n(t,x) \cdot D\psi^n = 0\ \  {\rm in}\ \ (t_0^n,\bar t)\times \R^d,\\
\psi^n(\bar t)= \bar \psi\ \  {\rm in}\ \  \R^d.
\end{array}\right.
\ee
It follows from the standard parabolic regularity that $\psi^n$ is bounded in $C^{(2+\delta)/2,2+\delta}$ by $C\|\bar \psi\|_{C^{2+\delta}}$ with  $C$ is independent of $n$ and $\bar t$, and since, in view of  the duality, we have 
$$
\lg \tilde \rho^{n,1}(\bar t), \bar \psi\rg =  \frac{1}{\lambda^n}  \big( \sigma^n \lg  \xi^n,  \psi^n(t_0^n) \rg -\sigma^n  \int_{t_0^n}^{\bar t}\lg R^{2,n}, D\psi^n\rg dt \big),
$$
we obtain 
$$
\sup_{t\in [t_0^n,T]} \|\tilde \rho^{n,1}(t)\|_{(C^{2+\delta})'} \leq  \frac{C}{\lambda^n} \big (\|\xi^n\|_{(W^{1,\infty})'} + \sup_{t\in [t_0^n,T]}\|R^{2,n}(t)\|_{(W^{1,\infty})'}\big ).
$$
Hence, $\sup_{t\in [t_0^n,T]} \|\tilde \rho^{n,1}(t)\|_{(C^{2+\delta})'} \to 0$ as $\dfrac{1}{\lambda^n} (\|\xi^n\|_{(W^{1,\infty})'} + \sup_{t\in [t_0^n,T]} \|R^{2,n}(t)\|_{(W^{1,\infty})'}) \to  0$. 
\vs
 
It follows from \eqref{takis38} that,  for any $\bar t^n\in [t_0^n,T]$ and ${\bar \psi}^n\in C^{2+\delta}$, if $\psi^n$ is the solution to \eqref{eq.dualeqBIS} on $[t_0^n,\bar t^n]$ with terminal condition $\psi^n(\bar t^n)={\bar \psi}^n$, then 
\be\label{repeho2n}
\lg \tilde \rho^{n,2}(\bar t^n), {\bar \psi}^n\rg =  - \sigma^n  \int_{t_0^n}^{\bar t^n}\int_{\R^d} \Gamma^n D\tilde z^n\cdot D\psi^n m^n. 
\ee
In order to prove the uniform convergence of the $ \tilde \rho^{n,2}$'s  in $(C^{2+\delta})'$, we assume that the $\|{\bar \psi}^n\|_{C^{2+\delta}}$'s  are bounded and, without loss of generality, that the $\bar t^n$'s and ${\bar \psi}^n$'s converge respectively to $\bar t\in [t_0,T]$ and $\bar \psi\in  C^{2+\delta}$,  the last convergence being in $C^{2+\delta_1}_{loc}$ for any $\delta_1\in (0,\delta)$. We need to prove that the $(\lg \tilde \rho^{n,2}(\bar t^n), {\bar \psi}^n\rg)$'s converge to $\lg \tilde \rho^{2}(\bar t), {\bar \psi}\rg$, where $\tilde \rho^2$ is the solution to 
$$
\partial_t\tilde  \rho^2-\Delta \tilde  \rho^2 -{\rm div}(V\tilde  \rho^2)  -\sigma{\rm div}(\Gamma D\tilde zm) = 0 \ \ \text{and} \ \ 
\tilde  \rho^2(t_0)= 0.
$$
Note that 
$$
\lg \tilde \rho^{2}(\bar t), {\bar \psi}\rg =  - \sigma  \int_{t_0}^{\bar t}\int_{\R^d} \Gamma D\tilde z\cdot D\psi m, 
$$
where $\psi$ is the solution to 
\[
%\left\{\begin{array}{l}
-\partial_t \psi -\Delta \psi + V(t,x) \cdot D\psi = 0\ \  {\rm in}\ \ (t_0,\bar t)\times \R^d \ \ \text{and} \ \ \psi(\bar t)= \bar \psi\ \  {\rm in}\ \  \R^d,
%\psi(\bar t)= \bar \psi\ \  {\rm in}\ \  \R^d.
%\end{array}\right.
\]
and recall that the $\Gamma^n$'s  and $D\tilde z^n$'s  are bounded and  converge locally uniformly to $\Gamma$ and to $D\tilde z$ respectively. 
\vs
Similarly, due to the parabolic regularity, the $\psi^n$'s are  bounded in $C^{(2+\delta)/2,2+\delta}$ and the  $D\psi^n$'s  converge locally uniformly to $D\psi$. 
\vs

Moreover, since   $m^n$ is the solution to  
 $$
 \partial_t m^n-\Delta m^n+{\rm div}(V^nm^n)=0 \ \ \text{in} \ \ (t_0^n, T] \ \ \text{and} \ \ m^n( t_0^n,\cdot)=m^n_0 \ \ \text{in} \ \R^d,
 $$
with $V^n$ uniformly bounded, we know that  the $m^n$'s  converge uniformly to $m$ in $\Pk$, and we have the  second-order moment estimate
$$
\sup_{t,n} \int_{\R^d}|x|^2 m^n(t,dx)\leq C\sup_{n} \int_{\R^d}|x|^2 m^n_0(dx) \leq C.
$$
In addition, using that $\Gamma D\tilde z\cdot D\psi$ is globally Lipschitz, for any $R\geq 1$ we find 
\begin{align*}
\left| \lg \tilde \rho^{n,2}(\bar t^n), {\bar \psi}^n\rg- \lg \tilde \rho^{2}(\bar t), {\bar \psi}\rg\right| & \leq 
C(|\sigma^n-\sigma|+|\bar t^n-\bar t|+|t_0^n-t_0|) \\
&\qquad + \sigma  \left| \int_{t_0\vee t_0^n}^{\bar t\wedge \bar t^n}\int_{\R^d} (\Gamma D\tilde z\cdot D\psi m-
\Gamma^n D\tilde z^n\cdot D\psi^n m^n)\right|  \\
&\leq C(|\sigma^n-\sigma|+|\bar t^n-\bar t|+|t_0^n-t_0|+\sup_t {\bf d}_1(m^n(t),m(t))+R^{-2}) \\
&\qquad + \sigma  \left| \int_{t_0\vee t_0^n}^{\bar t\wedge \bar t^n}\int_{B_R} (\Gamma D\tilde z\cdot D\psi -
\Gamma^n D\tilde z^n\cdot D\psi^n) m^n\right| .
\end{align*}
Letting $n\to \infty$ and then $R\to \infty$ proves the convergence of the  $\lg \tilde \rho^{n,2}(\bar t^n), {\bar \psi}^n\rg $'s  to $\lg \tilde \rho^{2}(\bar t), {\bar \psi}\rg$. It follows that  the sequence $(\tilde \rho^n)_{n\in \N}$
%$(\tilde \rho^n=\tilde \rho^{1,n}+\tilde \rho^{2,n})_{n\in \N}$  
converges to $\tilde \rho=\tilde \rho^2$  strongly in $C^0([t_0,T], (C^{2+\delta})')$.
\vs

To summarize the above, we know that the sequences  $(\sigma^n)_{n\in \N}$, $(\tilde z^n)_{n\in \N}$, $(\tilde \rho^n)_{n\in \N}$ and $(V^n)_{n\in \N}$ converge respectively  to  $\sigma\in [0,1]$,  $\tilde z$ in $C^{1,2}_{loc}$, $\tilde \rho$ in $C^0([0,T], (C^{2+\delta})')$ and $V$ in $C^0([t_0,T], C^{2+\delta}_{loc})$. 
\vs
Passing to the limit in \eqref{pierre} we infer that  $(\tilde z,\tilde \rho)$ is a  solution to the homogenous equation \eqref{eq.LSH}. 
\vs

 Since  $\sup_{t\in [t_0^n,T]} \|\rho^n(t)\|_{(C^{2+\delta})'}=1$ for any $n$, it follows  that $\sup_{t\in [t_0,T]} \|\rho(t)\|_{(C^{2+\delta})'}=1$. Thus  $(\tilde z,\tilde \rho)$ is a nonzero solution to the homogenous equation \eqref{eq.LSH} which  contradicts the strong stability assumption \eqref{takis31}.

 \vs
The second part of the proof consists in upscaling the regularity obtained in the first part. For this, we let 

$(V', t_0',m_0', \Gamma', R^{1,'}, R^{2,'}, R^{3,'}, \xi',\sigma')$ be such that \eqref{cond.data-Appen} holds and $(z',\rho')$ be a solution to \eqref{eq.LS} associated with these data, where $m'$ is the solution to \eqref{eq.m-Appen} with drift $V'$ and initial condition $m_0'$ at time $t_0'$.
\vs

We have already established that, for the  $M'$ in \eqref{def.Mprime-Appen}, 
$$
\sup_{t\in [t_0',T]} \|\rho'(t)\|_{(C^{2+\delta})'}\leq CM'. 
$$
It then follows from Lemma \ref{lem.apriori-Appen} that 
$$
\|z'\|_{C^{(2+\delta)/2,2+\delta}}+ \sup_{t'\neq t} \frac{\|\rho'(t',\cdot)-\rho'(t,\cdot)\|_{C^{2+\delta}}}{|t'-t|^{\delta/2}} \leq C(M'+ \sup_{t\in [t_0',T]} \|\rho'(t)\|_{(C^{2+\delta})'})\leq CM'.
 $$
\end{proof}

\subsection*{The stability property} We discuss here the notion of stabllity of a solution $(u,m)$ of the MFG-system arising in MFC. 
\vs

Let $(t_0, m_0)\in [0,T]\times \Pw$ and  $(m, \alpha)$ be a minimizer for $\mathcal U(t_0,m_0)$ with associated multiplier $u$,  that is,  the pair $(u,m)$ solves \eqref{MFG} and $\alpha(t,x)=-H_p(x,Du(t,x))$.

\begin{defn} The solution $(u,m)$ is strongly stable (resp. stable), if for all $\sigma\in [0,1]$ (resp. $\sigma=1$) the only solution $(z,\mu)\in C^{(1+\delta)/2, 1+\delta} \times  C^0([t_0,T]; (C^{2+\delta}(\R^d))')$ to
 the linearized system 
\be\label{MFGL}
\left\{\begin{array}{l}
\ds -\partial_t z -\Delta z +H_p(x,Du)\cdot Dz= \frac{\delta F}{\delta m}(x,m(t))(\mu(t)) \ \ {\rm in} \ \  (t_0,T)\times \R^d,\\[2mm]
\ds \partial_t \mu -\Delta \mu - {\rm div}(H_p(x,Du)\mu)-\sigma{\rm div}(H_{pp}(x,Du)Dz m)=0 \ \  {\rm in} \ \  (t_0,T)\times \R^d,\\[2mm]
\ds \mu(t_0)=0 \ \ \text{and} \ \  z(T, x)=\frac{\delta G}{\delta m}(x,m(T))(\mu(T)) \ \ {\rm in} \ \   \R^d,
\end{array}\right.
\ee
is the pair $(z,\mu)=(0,0)$.
\end{defn} 
Since, given a minimizer $(m,\alpha)$, the relation $\alpha= -H_p(x,Du)$ defines $Du$ uniquely, the stability condition depends on $(m,\alpha)$ only. We say that the minimizer $(m,\alpha)$ is strongly stable (resp. stable) if $(u,m)$ is strongly stable (resp. stable). 
\vs

The above makes also clear the definition of the regularity set $\ms O$ in \eqref{takis20}.  We remark that at this point we do not know whether $\ms O$ is a nonempty set. This will follow from Lemma~\ref{thm.Oopendense} below.
\vs

We also note  that \eqref{MFGL} is the linearized system studied in the previous subsection for the particular choice of vector field $V(t,x)= H_p(x,Du(t,x))$ and matrix $\Gamma(t,x)=H_{pp}(x,Du(t,x))$. To emphasize that we are working with this particular system and also be consistent with other references, heretofore we use the notation $(z,\mu)$ instead of $(z,\rho).$
\vs

The following lemma asserts that the minimizers starting from an initial condition in $\mathcal O$ are actually strongly stable. 

\begin{lem}\label{lem.StrongStab} Assume \eqref{ass.main}, fix $(t_0,m_0)\in \mathcal O$ and let $(m,\alpha)$ be the unique stable minimizer associated to $\mathcal U(t_0,m_0)$. Then $(m,\alpha)$ is strongly stable. 
\end{lem}

\begin{proof} Since, if $\sigma=1$, the claim is just the assumed stability of $(m,\alpha)$, in what follows we assume that   $$\sigma\in [0,1).$$

It follows from  Lemma \ref{lem.estLS-Appen} that $z \in C^{(2+\delta)/2,2+\delta}$, while using the duality, we have, for any $t\in [t_0,T]$,  
$$
\lg z(t,\cdot),\mu(t)\rg= -\int_{t_0}^t (\int_{\R^d} \Bigl(\sigma H_{pp}(x,Du) Dz\cdot dz m dx) + \lg \frac{\delta F}{\delta m}(\cdot,m(t))(\mu(t)),\mu(t)\rg\Bigr)dt,
$$
and, in particular, for  $t=T$, we get 
\be\label{ukzajenrsd}
\begin{split}
&\int_{t_0}^T \Bigl(\int_{\R^d} \sigma H_{pp}(x,Du) Dz\cdot Dz m dx + \lg \frac{\delta F}{\delta m}(\cdot,m(t))(\mu(t)),\mu(t)\rg \Bigr)dt\\
& \qquad  + 
\lg  \frac{\delta G}{\delta m}(\cdot,m(T))(\mu(T)),\mu(T)\rg =0.
\end{split}
\ee
Using that  $(m,\alpha)$ is a minimizer as well as  the second-order condition \eqref{ineq.SOC} with $(\rho, \beta)= (\mu, \sigma H_{pp}(x,Du)Dz)$ and recalling that $L_{\alpha,\alpha}(x,\alpha(t,x))H_{pp}(x,Du(t,x))= Id$, we get  
\begin{align*}
& \int_{t_0}^T \int_{\R^d} \Bigl(\sigma^2 H_{pp}(x,Du) Dz\cdot Dzm \\  
& \qquad +\int_{\R^d} \frac{\delta F}{\delta m}(x,m(t),y)\mu(t,x)\mu(t,y)dy\Bigr)dxdt + \int_{\R^{2d}} \frac{\delta G}{\delta m}(x,m(T),y)\mu(T,x)\mu(T,y)dydx\geq 0,
\end{align*}
and, in view of   \eqref{ukzajenrsd}, 
$$
(\sigma -\sigma^2) \int_{t_0}^T \int_{\R^d}  H_{pp}(x,Du) Dz\cdot Dz\ m\leq 0. 
$$
Since $H_{pp}>0$ and $\sigma -\sigma^2>0$, the last inequality yields  $Dz\ m=0,$ from which we easily conclude, going back to the equations satisfied by $\mu$ and by $z$,  that $(z, \mu)=(0,0)$.
\end{proof}

We turn next to $\ms O$. The next lemma  establishes an important property together with the fact $\mathcal O$ is not empty. 
A similar statement is proved in \cite{BrCa} when the state space is the torus. The adaptation to the whole space is given here for the sake of completeness.

\begin{lem}\label{lem.Odense} Assume \eqref{ass.main}. Fix  $(t_0,m_0)\in [0,T)\times \Pw$ and let $(m,\alpha)$ be a minimizer for $\mathcal U(t_0,m_0)$. Then $(t,m(t))$ belongs to $\mathcal O$ for any $t\in (t_0,T)$.
\end{lem}

\begin{proof} Fix $(t_0,m_0)\in [0,T)\times \Pw$, and let $(m,\alpha)$ be a minimizer for $\mathcal U(t_0,m_0)$ and  $u$ its associated multiplier. 
\vs
For  $t_1\in (t_0,T)$, set $m_1=m(t_1)$ and let $(\tilde m, \tilde \alpha)$ be an optimal solution for $\mathcal U(t_1,m_1)$ with associated multiplier $\tilde u$. Since, in view of the dynamic programming principle, 
$$
(\hat m, \hat \alpha) = \left\{\begin{array}{ll}
(m,\alpha) & {\rm on }\; [t_0,t_1)\times \R^d,\\[1.5mm]
(\tilde m, \tilde \alpha) & {\rm on}\; [t_1,T]\times \R^d,
\end{array}
\right.
$$

is  optimal for $\mathcal U(t_0,m_0)$, we know from Lemma~\ref{lem.reguum} that $\hat \alpha \in C^{(1+\delta)/2, 1+\delta}$. It follows  that $\alpha(t_1,\cdot)= \tilde \alpha(t_1,\cdot)$ and thus that $Du(t_1,\cdot)= D\tilde u(t_1,\cdot)$. Thus, the pair $((z^k)_{k=1, \dots, d},\mu)= ((\partial_{x_k}(u-\tilde u))_{k=1, \dots, d}, m-\tilde m)$ solves the system 
\be\label{eq.systLM}
\begin{split}
& -\partial z^k -\Delta z^k +g^k(t,x)= 0 \ \ {\rm in}\ \  (t_1,T)\times \R^d,\\ 
& \partial \mu -\Delta \mu +{\rm div}(h)= 0  \ \ {\rm in} \ \ (t_1,T)\times \R^d,\\
& \mu(t_1)=0 \ \ z^k(t_1, \cdot)= 0  \ \ {\rm in}\ \ \R^d,
\end{split}
\ee

where 
\begin{align*}
g^k(t,x)& =  H_{x_k}(x,Du)-H_{x_k}(x,D\tilde u)+H_p(x,Du)\cdot D(\partial_{x_k} u)-H_p(x,D\tilde u)\cdot D(\partial_{x_k}\tilde u)\\
& \qquad\qquad\qquad -F_{x_k}(x,m(t))+F_{x_k}(x,m^1(t)), \\
 h&=H_p(Du)m-H_p(D\tilde u)\tilde m.
\end{align*}
In order to estimate $g^k$ and $h$, we  note that, since $t_1>t_0$, $m, \tilde m \in  C^{1,2}([t_1,T]\times \R^d)$ and $m(t,\cdot)$ and $\tilde m(t,\cdot)$ are bounded in $L^2$. It follows that 
\begin{align} \label{kuqyesrdnfgc}
&\sum_{k=1}^d |g^k(t,x)|^2 \leq C(|z(t,x)|^2+|Dz(t,x)|^2+\|\mu(t)\|_{L^2}^2), \notag\\
&|h(t,x)|^2\leq C (|z(t,x)|^2+|\mu(t,x)|^2),\\[1.2mm]
& |{\rm div}(h(t,x))|^2\leq C (|z(t,x)|^2+|Dz(t,x)|^2+|\mu(t,x)|^2+|D\mu(t,x)|^2). \notag
\end{align}
Then a  Lions-Malgrange-type argument shows that $z_k=\mu=0$, and, hence,  the solution starting from $(t_1,m_1)$ is unique.  We refer to Lions and Malgrange  \cite{LiMa} for the original argument and Cannarsa and Tessitore \cite{CaTe} and \cite{BrCa} for  its adaptation to forward-backward equations. 
\vs
Next we check that this solution is stable. Let $(z,\mu)$ be a solution to \eqref{MFGL}
in  $[t_1,T]\times \R^d$ with $\sigma=1$, which by the standard  parabolic regularity is actually classical. An elementary calculation yields 
\be\label{tildeJDz}
\begin{split}
&\int_{t_1}^T\left(\int_{\R^d} H_{pp}(x,Du(t,x))Dz\cdot Dz\ m dx+\lg  \frac{\delta F}{\delta m}(\cdot,m(t)), \mu(t),\mu(t)\rg \right)dt\\[1.2mm] 
& \hskip1in  +\lg \frac{\delta G}{\delta m}(\cdot,m(t))(\mu(t)),\mu(t) \rg  =0. 
\end{split}
\ee
Using  Lemma~\ref{lem.secondordercond}, we know that, for any $\beta\in L^\infty([t_0,T]\times \R^d;\R^d)$ vanishing near $t=t_0$, if $\rho$ is the solution in the sense of distributions to \eqref{eq.forsecondorder} in $[t_0,T]\times \R^d$, then 
\begin{align*}
\tilde J(\beta) = & \int_{t_0}^T \Bigl( \int_{\R^d} L_{\alpha,\alpha}(x,\alpha(t,x)) \beta(t,x)\cdot   \beta(t,x)m(t,dx) +\lg  \frac{\delta F}{\delta m}(\cdot,m(t),\cdot), \rho(t)\otimes \rho(t)\rg\Bigr)dt \notag\\  
& \qquad + \lg \frac{\delta G}{\delta m}(\cdot,m(T),\cdot), \rho(T,\cdot)\otimes \rho(T,\cdot)\rg\geq 0.
\end{align*}
The solution $\bar \rho$ to \eqref{eq.forsecondorder} associated to the map $\bar \beta$ defined by $\bar \beta =0$ on $[t_0,t_1)$ and $\bar \beta= -H_{pp}(x,Du)Dz$ on $[t_1,T]$ is given by $\bar \rho(t)=0$ on $[t_0,t_1)\times \R^d$ and $\bar \rho(t)= \mu(t)$ on $[t_1,T]\times \R^d$. 
\vs
It then follows from \eqref{tildeJDz} that $\tilde J(\bar \beta)=0$, and, hence, $\bar \beta$ is a minimizer for $\tilde J$, a fact that,  by standard arguments (see, for example,  \cite{BrCa})), implies that $\beta$ is a  continuous map. Thus $Dz(t_1,\cdot)=0$. 
\vs

We differentiate with respect to space variable the first equation in \eqref{MFGL} and obtain that \\$((\partial_{x_k}z)_{k=1,\dots, d}, \mu)$ solves a system of the form \eqref{eq.systLM} with zero initial condition and data $g$ and $h$ satisfying \eqref{kuqyesrdnfgc}. 
This implies, as before,  that $((\partial_{x_k}z)_{k=1,\dots, d}, \mu)=(0,0)$. Coming back to \eqref{MFGL}, we obtain $z=\mu=0$. Therefore the solution is stable. 
\end{proof}

The next theorem establishes the key  properties of $\ms O$.

\begin{thm}\label{thm.Oopendense}  Assume \eqref{ass.main}.  The set $\mathcal O$ is open and dense in $[0,T)\times \Pw$.  
\end{thm}

\begin{proof}
Lemma \ref{lem.Odense}   implies that the set  $\mathcal O$ is a nonempty, dense subset of  $[0,T)\times \Pw$. 
\vs

Next we show that  $\mathcal O$ is open arguing  by contradiction. For this, we fix $(t_0,m_0)\in \mathcal O$ and assume that there are initial positions $(t^n,m^n_0)\notin \mathcal O$ which converge to $(t_0, m_0)$  (in $\Pw$ for $(m^n_0)_{n\in \N}$).  
Let $(m,\alpha)$ be the unique and stable minimizer for $(t_0,m_0)$ and $u$ be the associated multiplier, that is,  $\alpha= -H_p(x,Du)$ and the pair $(u,m)$ solves \eqref{MFG}. 
\vs
Since  $(t^n,m^n_0)\notin \mathcal O$, there are two cases (up to subsequences): either, for all $n$,  there exist several minimizers for $\mathcal U(t^n, m^n_0)$ or, for all $n$, there exists a unique minimizer which  is not stable. This latter case is ruled out  by Lemma \ref{lem.estLS-Appen} and the strong stability of $(m,\alpha)$. 
\vs
It remains to consider the first case and we argue as follows. Let  $(m^{n,1},\alpha^{n,1})$ and $(m^{n,2},\alpha^{n,2})$ be two distinct minimizers starting from $(t^n, m^n_0)$ with associated multipliers  $u^{n,1}$ and $u^{n,2}$ respectively. 
\vs
Since the problem with initial condition $(t_0,m_0)$ has a unique minimizer, it follows from Lemma \ref{lem.stabilo} that, for $i=1,2$,     the $(m^{n,i},\alpha^{n,i})$'s  converge to $(m,\alpha)$ in $C^0([0,T]; \Pk)\times C^{\delta/2\delta}$ while the $u^{n,i}$'s, $Du^{n,i}$'s and $D^2u^{n,i}$'s converge to $u$, $Du$ and $D^2u$ respectively in $C^{\delta/2,\delta}$. 
\vs
Set 
$$
\theta^n = \|Du^{n,1}-Du^{n,2}\|_{C^{\delta/2,\delta}}+\sup_{t\in [t_0,T]} \dk(m^{n,1}(t),m^{n,2}(t))
$$ and note that, since  $(m^{n,1},\alpha^{n,1})$ and $(m^{n,2},\alpha^{n,2})$ are distinct, $\theta^n>0$ and, in view of the previous discussion, $\theta_n\to 0$, and,
finally, 
\be\label{lizekjsrdfkekrf}
\theta^n \leq C \|Du^{n,1}-Du^{n,2}\|_{C^{\delta/2,\delta}}. 
\ee
This  last estimate follows from the fact that again the uniform parabolicity implies that the  $D^2u^{n,i}$'s are uniformly bounded. Applying Gronwall's inequality to the stochastic differential equations associated with the Kolmogorov equations satisfied by $m^{n,1}$ and $m^{n,2}$, we find the distance \\
$\sup_{[t_n, T]}  \dk(m^{n,1}(t),m^{n,2}(t))$ is controlled by $ C \|Du^{n,1}-Du^{n,2}\|_\infty$ and thus by $C \|Du^{n,1}-Du^{n,2}\|_{C^{\delta/2,\delta}}$. 
\vs
Next we introduce the differences   $z^n= (u^{n,1}-u^{n,2})/\theta^n$, $\mu^n= (m^{n,1}-m^{n,2})/\theta^n$ and observe that 
$$
\left\{\begin{array}{l}
\ds -\partial_t z^n -\Delta z^n +H_p(x,Du^{n,1})\cdot Dz^n= \frac{\delta F}{\delta m}(x,m^{n,1}(t))(\mu^n(t))+R^{n,1},\\[1.5mm]
\ds \partial_t \mu^n -\Delta \mu^n -{\rm div}(H_p(x,Du^{n,1})\mu^n)-{\rm div}(H_{pp}(x,Du^{n,1})Dz^n m^{n,1})={\rm div}(R^{2,n}),\\[1.3mm]
\ds \mu^n(t_0)=0, \; z^n(T, x)=\frac{\delta G}{\delta m}(x,m^{n,1}(T))(\mu^n(T))+R^{3,n}
\end{array}\right.
$$
with 
\begin{align*}
R^{n,1} = & (\theta^n)^{-1} \Bigl[H(x,Du^{n,2})-H(x,Du^{n,1})-H_p(x,Du^{n,1})\cdot (Du^{n,2}-Du^{n,1})\\
 & \qquad\qquad -\Bigl( F(x,m^{n,2}(t))-F(x,m^{n,1}(t))- \frac{\delta F}{\delta m}(x,m^{n,1}(t))(m^{n,2}(t)-m^{n,1}(t)) \Bigr)\bigr],
\end{align*}
\begin{align*}
R^{n,2}= &  -(\theta^n)^{-1} \Bigl[ H_p(x,Du^{n,2})m^{n,2}- H_p(x,Du^{n,1})m^{n,1}-H_p(x,Du^{n,1})(m^{n,2}-m^{n,1})\\
& \qquad \qquad -H_{pp}(x,Du^{n,1})(Du^{n,2}-Du^{n,1})m^{n,1})\Bigr],
\end{align*}
and
$$
R^{n,3}=-(\theta^n)^{-1} \Bigl[ G(x,m^{n,2}(T))-G(x,m^{n,1}(T))- \frac{\delta G}{\delta m}(x,m^{n,1}(T))(m^{n,2}(T)-m^{n,1}(T)) \bigr].
$$
\vs
It follows from the regularity of $F$, $G$ and $H$ and the definition of $\theta^n$ that 
\be\label{estiR1R2}
\|R^{n,1}\|_{C^{\delta/2,\delta}}+\|R^{n,3}\|_{C^{2+\delta}} + \sup_{t\in [t_0,T]} \|R^{n,2}(t)\|_{(W^{1,\infty})'} \leq C\theta^n.  
\ee
However, Lemma \ref{lem.estLS-Appen} yields that  the sequence $(z^n)_{n\in \N}$  tends to $0$ in $C^{(1+\delta)/2,1+ \delta}$, which contradicts \eqref{lizekjsrdfkekrf}.
\end{proof}

\subsection*{The smoothness of $\mathcal U$ in $\mathcal O$} We prove here Theorem~\ref{thm.mainreguU0}. 
\vs
Before we present the arguments, we state below as lemma a preliminary fact that is needed to establish the regularity of $\ms U$. It is about   a stability property in the appropriate norms for the multipliers associated with minimizers starting in $\ms O$. In turn, this will allow us to compute the derivative of $\ms U$ with respect to $m$. Its proof is presented at the end of this subsection.
\vs

\begin{lem}\label{lem.Lip} Assume \eqref{ass.main} and fix $(t_0,m_0)\in \mathcal O$. There exists $\delta, C>0$ such that, for any $t_0'$, $m_0^1,m_0^2$ satisfying  $(t_0',m_0^i)\in \mathcal O$, $|t_0'-t_0|<\delta$, ${\bf d}_2(m_0,m_0^i)<\delta$, if $(m^i,\alpha^i)$ is the unique minimizer starting from $(t_0', m_0^i)$ with associated multiplier $u^i$ for $i=1$ and $i=2$,  then 
$$
\|u^1-u^2\|_{C^{(2+\delta)/2,2+\delta}}+ \sup_{t\in [t_0',T]} {\bf d}_2(m^1(t),m^2(t)) \leq C {\bf d}_2(m_0^1,m^2_0).
$$
\end{lem}

We continue with the proof of Theorem~\ref{thm.mainreguU0} which consists of three parts. In the first, we establish the regularity of $\ms U$ in $\ms O$. In the second, we show that the infinite dimensional Hamilton-Jacobi equation \eqref{HJdsO0} is satisfied in $\ms O$. Finally, the third part is about \eqref{takis50}. 

\begin{proof}[Proof of Theorem \ref{thm.mainreguU0}] {\it Part 1: The regularity of $\ms U$.}
Lemma \ref{lem.estiVN} yields the Lipschitz continuity of $\mathcal U$. 
\vs

We establish  that $\mathcal U$ is differentiable at any $(t_0,m_0)\in \mathcal O$. 
We fix such a $(t_0,m_0)$. Let  $(m,\alpha)$ be the unique stable minimizer  for $\mathcal U(t_0,m_0)$ and $u$ its  associated multiplier. We  check that $D_m\mathcal U(t_0,m_0,\cdot)$ exists and is given by $Du(t_0,\cdot)$. 
\vs

Let $\delta>0$ and $\delta'\in (0,\delta)$ be such that  the  $\delta-$neighborhood of the $[0,T]\times \Pw$-compact set $\{(t,m(t)): t\in [t_0,T]\}$  is contained in $\ms O$, and, 
for any $m^1_0\in B(m_0,\delta')$, $\sup_{t\in [t_0,T]} {\bf d_1}(m(t),m^1(t))<\delta$, where $(m^1,\alpha^1)$ is the minimizer for $\mathcal U(t_0, m_0^1)$.
\vs 

Fix $m^1_0\in B(m_0,\delta')$.  Let  $(m^1,\alpha^1)$ be the minimizer for  $\mathcal U(t_0, m_0^1)$, $u^1$ its associated multiplier, $(z,\mu)$  the solution of the linearized system \eqref{MFGL} with initial condition $\mu (0)= m^1_0-m_0$, set 
$$
(w,\rho)= (u^1-u-z, m^1-m-\mu)
$$
and note that $(w,\rho)$ satisfies the linearized system \eqref{eq.LS} with $\xi= 0$,
$$
R^1= -(H(x,Du^1)-H(x,Du)-H_p(x,Du)\cdot (Du^1-Du)) +F(x,m^1)-F(x, m) -\frac{\delta F}{\delta m}(x,m(t))(m^1(t)-m(t)), 
$$
\begin{align*}
R^2=& H_p(x,Du^1)m^1-H_p(x,Du)m -H_p(x,Du)(m^1-m)-H_{pp}(x,Du)\cdot (Du^1-Du)m,
\end{align*}
and 
$$
R^3=G(x,m^1)-G(x, m) -\frac{\delta G}{\delta m}(x,m(t))(m^1(t)-m(t)). 
$$
Then, using   Lemma~\ref{lem.Lip}, we get 
$$
\|R^1\|_{C^{\delta/2,\delta}}+\|R^{3}\|_{C^{2+\delta}}+ \sup_{t\in [t_0,T]}\|R^{2}(t)\|_{(W^{1,\infty})'} \leq C{\bf d}_1^2(m^1_0,m_0)
$$
and, in view of Lemma~\ref{lem.estLS-Appen},  
$$
 \|u^1-u-w\|_{C^{(2+\delta)/2,2+\delta}}+\sup_{t\in [t_0,T]} \|m^1(t)-m(t)-\mu(t)\|_{(C^{2+\delta})'} \leq C{\bf d}_1^2(m^1_0,m_0). 
$$
Recall that $\alpha^1= -H_p(x, Du^1)$. Thus 
$$
\alpha^1= \alpha - H_{pp}(x,Du). Dw+ o({\bf d}_1(m^1_0,m_0)).
$$
where $o(\cdot)$ is small in uniform norm. 
\vs
It follows that 
\begin{align*} 
 &\mathcal U(t_0,m^1_0) = \int_{t_0}^T \Bigl( \int_{\R^d} L(x,\alpha^1) m^1 +\mathcal F(m^1)\Bigr)dt + \mathcal G(m^1(T)) \\
&\qquad = \mathcal U(t_0,m_0) + \int_{t_0}^T  \int_{\R^d} \Bigl(D_\alpha L(x,\alpha)\cdot(-H_{pp}(x,Du)Dw m +L(x,\alpha)\mu(t,x)  +F(x,m(t))\mu(t,x)\Bigr)dxdt \\
& \qquad \qquad + \int_{\R^d} G(x,m(T))\mu(T,x)dx +  o({\bf d}_1(m^1_0,m_0)).
\end{align*}
On the other hand, recalling the equations satisfied by $u$ and $\mu$ and using duality we find 
\begin{align*}
& \int_{\R^d}G(x,m(T))\mu(T,x)dx - \int_{\R^d} u(t_0,x)(m^1_0-m_0)(dx)\\
& \qquad = \int_{t_0}^T \int_{\R^d} \Bigl( (H(x,Du) -F(x,m(t))-H_p(x,Du)\cdot Du)\mu -H_{pp}(x,Du)Du\cdot Dw m \Bigr). 
\end{align*}
Thus,
\begin{align*} 
 &\mathcal U(t_0,m^1_0) =  \mathcal U(t_0,m_0) +  \int_{\R^d} u(t_0,x)(m^1_0-m_0)(dx) \\
 & \qquad + \int_{t_0}^T  \int_{\R^d} \Bigl((H(x,Du)-H_p(x,Du)\cdot Du +L(x,\alpha))\mu \\
& \qquad -(H_{pp}(x,Du)Du\cdot Dw  +  D_\alpha L(x,\alpha)\cdot(H_{pp}(x,Du)Dw)) m    \Bigr)dxdt+  o({\bf d}_1(m^1_0,m_0))
\end{align*}
In view of the relationship (convex duality)  between $H$ and $L$ and the fact that $\alpha=-H_p(x,Du)$,  we have $D_\alpha L(x,\alpha)=-Du$ and, therefore,
$$
H(x,Du)-H_p(x,Du)\cdot Du +L(x,\alpha)=0
$$
and 
$$
H_{pp}(x,Du)Du\cdot Dw  +  D_\alpha L(x,\alpha)\cdot(H_{pp}(x,Du)Dw)=0.
$$
Thus, 
\begin{align*} 
 &\mathcal U(t_0,m^1_0) =  \mathcal U(t_0,m_0) +  \int_{\R^d} u(t_0,x)(m^1_0-m_0)(dx) +  o({\bf d}_1(m^1_0,m_0)).
\end{align*}
It follows that $\mathcal U(t_0,\cdot)$ has a linear derivative at $m_0$ given by $u(t_0,\cdot)$ and, hence, $$D_mU(t_0,m_0,x)= Du(t_0,x).$$
\vs

Recalling the stability of the map $m_0\to (u,m)$ proved in Lemma \ref{lem.stabilo}, we actually have that $(t_0,m_0)\to D_mU(t_0,m_0, \cdot)$ is continuous in $\mathcal O$ with the respect to  the ${\bf d}_2-$distance for the \\ measure variable into $C^2$. 

\vs

{\it Part 2: The Hamilton-Jacobi equation.}
Next we show that $\ms U$ is a classical solution to  \eqref{HJdsO0}. 
\vs
Using the notation of Part 1 and  the dynamic programming principle with $h>0$ small, we find 
$$
\mathcal U(t_0,m_0)= \int_{t_0}^{t_0+h}\Bigl( \int_{\R^d} L(x,\alpha(t,x))m(t,x)dx+ \mathcal F(m(t))\Bigr)dt + \mathcal U(t_0+h, m(t_0+h)),
$$
and, in view of the $C^1-$regularity of $\mathcal U$, 
\begin{align*}
\mathcal U(t_0+h, m(t_0+h))- \mathcal U(t_0+h, m_0) = &\int_{t_0}^{t_0+h} \int_{\R^d} \Bigl( {\rm Tr}D^2_{ym} \mathcal U(t_0+h, m(t),y)\\
&\;  
 + 
D_m\mathcal U(t_0+h, m(t),y)\cdot H_p(t,Du(t,y))\Bigr) m(t,dy)dy dt.
\end{align*}
It follows that  $\partial_t \mathcal U(t_0,m_0)$ exists and is given by 
\begin{align*}
\partial_t \mathcal U(t_0,m_0) 
=& -  \int_{\R^d} L(x,\alpha(t_0,x))m_0(dx)- \mathcal F(m_0) \\
& -
\int_{\R^d} \Bigl( {\rm Tr}D^2_{ym} \mathcal U(t_0, m_0,y)
 + 
D_m\mathcal U(t_0, m_0,y)\cdot H_p(t_0,Du(t_0,y))\Bigr) m_0(dy). 
\end{align*}
Since  $Du(t_0,x)= D_m\mathcal U(t_0,m_0,x)$ and $\alpha(t_0,x)=-H_p(x,Du(t_0,x))$, \eqref{HJdsO0} is then satisfied. 
\vs

{\it Part 3: The regularity of $D_m\ms U$.} We prove that \eqref{takis50} holds. 
\vs

Let $\delta,C>0$ be given by  Lemma \ref{lem.Lip}. For any $t\in (t_0-\delta, t_0+\delta)$ and $m^2_0,m^2_0\in \Pw$ with ${\bf d}_2(m_0,m^1_0)<\delta$, ${\bf d}_2(m_0,m^2_0)<\delta$, $(t,m^1_0)\in \mathcal O$ and $(t,m^2_0)\in \mathcal O$, and for any $x^1,x^2\in \R^d$, let  $u^1$ (respectively $u^2$) be the multiplier associated with the unique minimizer $(m^1,\alpha^1)$ for $\mathcal U(t,m^1_0)$ (respectively with the unique minimizer  $(m^2,\alpha^2)$ for $\mathcal U(t, m^2_0)$).
\vs
Since, as already established,   $D_m\mathcal U(t,m^1_0, x^1)=Du^1(t,x^1)$ and $D_m\mathcal U(t,m^2_0, x^2)=Du^2(t,x^2)$, we find, using  Lemma~\ref{lem.reguum} and Lemma~\ref{lem.Lip}, that 
\begin{align*}
\left| D_m\mathcal U(t,m^1_0, x^1)- D_m\mathcal U(t,m^2_0, x^2)\right| & \leq \left| Du^1(t, x^2)- Du^2(t,x^2)\right| +\|D^2u^1\|_\infty|x^1-x^2| \\ 
& \leq C ({\bf d}_2(m^1_0,m^2_0)+|x^1-x^2|).
\end{align*}
\end{proof}

We conclude with the remaining proof.

\begin{proof}[Proof of Lemma~\ref{lem.Lip}] Let $(m,\alpha)$ be the unique stable minimizer  starting from $(t_0,m_0)$ with multiplier $u$. 
It follows  from Lemma \ref{lem.StrongStab} that the associated linear system \eqref{MFGL} is strongly stable. 
\vs
We set $V= -H_p(x,Du)$ and $\Gamma= -H_{pp}(x,Du)$, consider  the neighborhood $\mathcal V$  (in the local uniform convergence) of $(V,\Gamma)$ given in Lemma \ref{lem.estLS-Appen}, and choose $\delta>0$ so that, for any $m_0^1$ such that  ${\bf d}_2(m_0,m^1_0)<\delta$, we have that $$(V^1, \Gamma^1)\in \mathcal V,$$ with $V^1=-H_p(x,Du^1)$ and $\Gamma_1= H_{pp}(x,Du^1)$, $u^1$ being the multiplier associated with the optimal solution $(m^1,\alpha^1)$ for $\mathcal U(t_0,m_0^1)$. The above is possible since, if $\delta$ is small, then, in view of  Lemma \ref{lem.stabilo}, $u^1$ is close to $u$ in $C^{\delta/2,\delta}$. 
\vs
Furthermore, choosing, if necessary,  $\delta$ even smaller, we have that,  for some $\eta>0$ to be chosen below, and,  for any  $t_0'$, $m_0^i$, $(m^i, \alpha^i)$ and $u^i$ as above such that $|t_0'-t_0|<\delta$ and ${\bf d}_2(m_0,m_0^i)<\delta$ for $i=1$ and $i=2$, 
\be\label{lzkej:snrdf,gc}
\|u^1-u^2\|_{C^{(2+\delta)/2,2+\delta}} + \sup_{t\in [t_0',T]} {\bf d}_2(m^1(t),m^2(t)) \leq \eta.
\ee

Classical estimates on the Kolmogorov equation (see, for instance, \cite{CDLL}) yield that,  for $t_0'$, $m_0^i$, $(m^i, \alpha^i)$ and $u^i$ as above, 
we have 
\be\label{zuesdzmesdf}
\sup_{t\in [t_0',T]} {\bf d}_2(m^1(t),m^2(t))\leq C( {\bf d}_2(m^1_0,m^2_0)+ \|Du^1-Du^2\|_\infty),
\ee
for a  constant $C$ depending on $T$, $H$ and  $\|D^2u^i\|_\infty$, which is uniformly bounded by Lemma \ref{lem.apriori-Appen}. 
\vs
Then the pair 
$$
(z,\mu)= (u^1-u^2,m^1-m^2)
$$
satisfies the linearized system \eqref{eq.LS} with 
%$V(t,x)= -H_p(x,Du^2)$, $\Gamma= -H_{pp}(x,Du^2)$, $\xi= m^1_0-m_0^2$, 
\begin{align*}
V(t,x)= -H_p(x,Du^2), \ \ \Gamma= -H_{pp}(x,Du^2), \ \ \xi= m^1_0-m_0^2, 
\end{align*}
\begin{align*}
R^1&= -(H(x,Du^1)-H(x,Du^2)-H_p(x,Du^2)\cdot (Du^1-Du^2)) +F(x,m^1(t))-F(x, m^2(t)) \\
& \qquad \qquad -\frac{\delta F}{\delta m}(x,m^2(t))(m^1(t)-m^2(t)), 
\end{align*}
\begin{align*}
R^2&= H_p(x,Du^1)m^1-H_p(x,Du^2)m^2 -H_p(x,Du^2)(m^1-m^2)-H_{pp}(x,Du^2)\cdot (Du^1-Du^2)m^2\\
&=(H_p(x,Du^1)-H_p(x,Du^2))(m^1-m^2)\\
&\qquad +(H_p(x,Du^1)-H_p(x,Du^2)-H_{pp}(x,Du^2)\cdot (Du^1-Du^2))m^2,
\end{align*}
and 
$$
R^3=G(x,m^1)-G(x, m^2) -\frac{\delta G}{\delta m}(x,m^2(t))(m^1(t)-m^2(t)). 
$$
Note that 
\begin{align*}
R^2=& (H_p(x,Du^1)-H_p(x,Du^2))(m^1-m^2)\\
& \qquad +(H_p(x,Du^1)-H_p(x,Du^2)-H_{pp}(x,Du^2)\cdot (Du^1-Du^2))m^2.
\end{align*}
Thus
\begin{align*}
M& =\|\xi\|_{(W^{1,\infty})'}+\|R^1\|_{C^{\delta/2,\delta}}+\|R^{3}\|_{C^{2+\delta}}+ \sup_{t\in [t_0',T]}\|R^{2}(t)\|_{(W^{1,\infty})'}\\
& \leq {\bf d}_1(m^1_0,m^2_0)+
C\{\|Du^1-Du^2\|_{C^{\delta/2,\delta}}^2+ \sup_t{\bf d}^2_{2}(m^1(t),m^2(t))\}.
\end{align*}
It follows from  Lemma \ref{lem.estLS-Appen} that 
$$
\|u^1-u^2\|_{C^{(2+\delta)/2, 2+\delta}}\leq C\{ {\bf d}_1(m^1_0,m^2_0)+ \|Du^1-Du^2\|_{C^{\delta/2,\delta}}^2+ \sup_t{\bf d}^2_{2}(m^1(t),m^2(t))\}.
$$
Hence, choosing $\eta>0$ small enough in \eqref{lzkej:snrdf,gc}, we find 
$$
\|u^1-u^2\|_{C^{(2+\delta)/2, 2+\delta}}\leq C\{ {\bf d}_1(m^1_0,m^2_0)+  \sup_t{\bf d}^2_{2}(m^1(t),m^2(t))\},
$$
and inserting the last  inequality in  \eqref{zuesdzmesdf} we obtain
$$
\sup_t{\bf d}_{2}(m^1(t),m^2(t)) \leq C\{ {\bf d}_2(m^1_0,m^2_0)+  \sup_t{\bf d}^2_{2}(m^1(t),m^2(t))\}, 
$$
which yields, for $\eta>0$ small enough, 
$$
\sup_t{\bf d}_{2}(m^1(t),m^2(t)) \leq C {\bf d}_2(m^1_0,m^2_0). 
$$
Going back to the previous inequality on $\|u^2-u^2\|_{C^{(2+\delta)/2, 2+\delta}}$ completes the proof. 
\end{proof}

\section{The propagation of chaos} \label{sec.proofThm2}

We present the proof  of Theorem \ref{thm.main20}, which consists of several steps  each of which is stated below as separate lemma. 
\vs
In preparation, we fix $(t_0,m_0)\in \mathcal O$ with $M_{d+5}(m_0)<+\infty$  and let $(m,\alpha)$ be the unique minimizer for $\mathcal U(t_0,m_0)$.  
\vs

It follows from Theorem~\ref{thm.mainreguU0} and the compactness of the curve $\{(t,m(t)) :  t\in [t_0,T]\}$ that there exists 
$\delta, C>0$ such that, for any $t_1\in [t_0,T]$,  $t\in (t_1-\delta, t_1+\delta)$ and $m^1_0,m^2_0\in \Pw$ with ${\bf d}_2(m(t_1),m^1_0)<\delta$, ${\bf d}_2(m(t_1),m^2_0)<\delta$, $(t_1,m^1_0)\in \mathcal O$ and $(t_1,m^2_0)\in \mathcal O$, and $x^1,x^2\in \R^d$, 
\be\label{DmULipsch}
|D_m\mathcal U(t_1,m^1_0,x^1)-D_m\mathcal U(t_1,m^2_0,x^2)|\leq C (|x^1-x^2|+ {\bf d}_2(m^1_0,m^2_0)).
 \ee
 For $\sigma\in (0,\delta)$,  set 
$$
 V_\sigma = \{(t,m')\in [t_0, T]\times \Pw :  {\bf d}_2(m',m(t))<\sigma\}
$$
and 
$$
V^N_\sigma= \{(t, {\bf x})\in [0,T]\times (\R^d)^N :  (t,m^N_{{\bf x}})\in  V_\sigma\}.
$$

We consider the solution $({\bf X}^N_t)_{t\in [t_0,T]}=(X^{N,1}_t, \ldots,X^{N,N}_t )_{t\in [t_0,T]}$ to 
\be\label{def.XNi}
dX^{N,j}_t= Z^j-\int_{t_0}^t H_p(X^{N,j}_s, D_m\mathcal U(s,m^N_{{\bf X^{N}_s}},X^{N,j}_s))ds +\sqrt{2}(B^j_s-B^j_{t_0}),
\ee
on the time interval $[t_0, \tau^N]$, where the stopping time $\tau^N$ is defined  by
$$
\tau^N=\inf\left\{t\in [t_0,T], \; (t,{\bf X}^N_t) \notin V^N_{\delta/2} \right\} \ \ \text{or} \ \ \text{$\tau^N=T$, if there is no such a $t$.}
$$
Note that, in view of \eqref{DmULipsch}, ${\bf X}^N$ is uniquely defined.  
\vs
Set 
\be\label{def.tildetau}
\tilde \tau^N=\inf\{ t\in [t_0,\tau^N], \; (t,{\bf Y}^N_t)\notin V^N_{\delta}\} \ \ \text{or} \ \ \text{$\tilde \tau^N=\tau^N$ is there is no such a $t$.}
\ee

\begin{lem} \label{lem.aqemzjlrsdk} Assume \eqref{ass.main}. There is a constant $C>0$ depending on $(t_0,m_0)$ such that  
\be\label{aqemzjlrsdk}
\E\left[\sup_{t\in [t_0,\tau^N]} {\bf d}_1(m^N_{{\bf X}^N_t}, m(t)) \right]\leq  C N^{-1/(d+8)}
\ee
and
$$
\P\left[\tau^N<T\right] \leq C N^{-1/(d+8)}.
$$
\end{lem}

\begin{proof} The proof is standard and relies on propagation of chaos estimates; see, for instance, the proof of Theorem 5.6 in \cite{DeLaRa20}. 
\vs

Let $\tilde X^{N,i}$ be the i.i.d.  solutions to 
$$
d\tilde X^{N,i}_t= -H_p(\tilde X^{N,j}_t, D_m\mathcal U(t,m(t),\tilde X^{N,i}_t) )dt +\sqrt{2}dB^j_t \ \ \ \tilde X^{N,i}_{t_0}=Z^i,
$$
which exist on $[t_0,T]$, in view of the global Lipschitz property of  $D_m\mathcal U(t,m(t),\cdot)$.
\vs

Then  \eqref{aqemzjlrsdk} is an easy consequence of the inequality 
$$
\E\left[\sup_{t\in [0,T]} {\bf d}^2_2(m^N_{\tilde {\bf X}^N_t}, m(t))\right]\leq CN^{-2/(d+8)},
$$
which follows from  Theorem 3.1 in Horowitz and Karandikar  \cite{HoKa}, for some $C$ depending only on $(t_0,m_0)$ through the regularity of $D_m\mathcal U$ in $V_\delta$.
\vs
Then 
$$
\P\left[\tau^N<T\right]  \leq \P\left[ \sup_{t\in [t_0,\tau^N]} {\bf d}_1(m^N_{{\bf X}^N_t}, m(t))\geq \delta/2\right] \leq C \delta^{-1} N^{-1/(d+8)}.
$$
\end{proof}

Next,  for $(t,\bx) \in V^N_\delta$,  we set 
$$
\mathcal U^N(t, {\bf x})= \mathcal U(t, m^N_{{\bf x}}).
$$
It is immediate that  $\mathcal U^N$ is $C^{1,1}$ on $V^N_\delta$ and (see, for example, \cite{CDLL}) 
$$
D_{x^j} \mathcal U^N(t, {\bf x})= \frac1N D_m \mathcal U(t, m^N_{{\bf x}}, x^j) \ \ \text{and} \ \ | D^2_{x^jx^j} \mathcal U^N(t, {\bf x})- \frac1N D^2_{ym} \mathcal U(t, m^N_{{\bf x}}, x^j)|\leq  \frac{C}{N^2}. 
$$
Finally,  $\mathcal U^N$ satisfies
\be\label{eq.OON}
\left\{\begin{array}{l}
\ds -\partial_t \mathcal U^N(t,{\bf x}) -\sum_{j=1}^N \Delta_{x^j}\mathcal U^N(t,{\bf x}) +O^N(t,{\bf x})
+\frac1N\sum_{j=1}^N H(x^j, N D_{x^j}\mathcal U^N(t, {\bf x}))=\mathcal F(m^N_{{\bf x}})\ \  {\rm in}\ \ V^N_\delta, \\
\ds  \mathcal U^N(T,{\bf x})= \mathcal G(m^N_{{\bf x}})\ \  {\rm on}\ \  (\R^d)^N,
\end{array}\right.
\ee
where 
\be\label{evalON}
|O^N(t,{\bf x})|\leq CN^{-1}\qquad a.e. .
\ee

\begin{lem}\label{lem.sldjkfncv}  Let ${\bf Y}^N=(Y^{N,i})_{i=1,\ldots, N}$ and $\tilde \tau^N$ be defined by \eqref{takis51} and \eqref{def.tildetau} respectively. Then, 
\begin{align*}
& \E\bigl[ \int_{t_0}^{\tilde \tau^N} N^{-1}\sum_j |H_p(Y^{N,j}_t, ND_{x^j}\mathcal U^N)-H_p(Y^{N,j}_t, ND_{x^j}\mathcal V^N)|^2dt \bigr]  \leq  C(N^{-1}+ R^N), 
\end{align*}
where $R^N=\|\mathcal U^N-\mathcal V^N\|_\infty$ and $\mathcal U^N$, $\mathcal V^N$  and their derivatives are evaluated at $(t, {\bf Y}^N_t)$. 
\end{lem} 

\begin{proof}
For  $t\in [t_0,\tilde \tau^N]$ and, in view of \eqref{evalON}, we have 
\begin{align*}
&d\mathcal U^N(t, {\bf Y}^N_t)= (\partial_t \mathcal U^N +\sum_j \Delta_{x^j} \mathcal U^N- \sum_j H_p(Y^{N,j}_t, ND_{x^j}\mathcal V^N)\cdot D_{x^j}\mathcal U^N)dt+\sqrt{2}\sum_j  D_{x^j}\mathcal U^N\cdot dB^j_t \\
& = (\frac1N \sum_j H(Y^{N,j}, N D_{x^j}\mathcal U^N(t, {\bf Y}^N_t)) - \sum_j H_p(Y^{N,j}_t, ND_{x^j}\mathcal V^N)\cdot D_{x^j}\mathcal U^N+ O^N -\mathcal F(m^N_{{\bf Y}^N_t}))dt\\
& \qquad +\sqrt{2}\sum_j  D_{x^j}\mathcal U^N\cdot dB^j_t \\
& \geq (\frac1N \sum_j (-L(Y^{N,j}, -H_p(Y^{N,j}_t, ND_{x^j}\mathcal V^N))+C^{-1} |H_p(Y^{N,j}_t, ND_{x^j}\mathcal U^N)-H_p(Y^{N,j}_t, ND_{x^j}\mathcal V^N)|^2) \\ 
& \qquad \qquad   -CN^{-1} -\mathcal F(m^N_{{\bf Y}^N_t}))dt +\sqrt{2}\sum_j  D_{x^j}\mathcal U^N\cdot dB^j_t,
\end{align*}
the inequality following from the uniform convexity of $H$ in bounded sets. 

\vs
We take expectations and integrate between $t_0$ and $\tilde \tau^N$ above to get 
\begin{align*}
&\E\left[\mathcal U^N(\tilde \tau^N, {\bf Y}^N_{\tilde \tau^N})\right] -\E\left[\mathcal U^N(t_0, {\bf Z}^N)\right] \\
& \geq \E\bigl[ \int_{t_0}^{\tilde \tau^N} (\frac1N \sum_j \Bigl(-L(Y^{N,j}, -H_p(Y^{N,j}_t, ND_{x^j}\mathcal V^N))+C^{-1} |H_p(Y^{N,j}_t, ND_{x^j}\mathcal U^N)-H_p(Y^{N,j}_t, ND_{x^j}\mathcal V^N)|^2\Bigr) \\ 
& \qquad \qquad   -CN^{-1} -\mathcal F(m^N_{{\bf Y}^N_t}))dt \bigr].
\end{align*}
Rearranging, using the definition of $R^N$, the dynamic programming principle and the optimality of ${\bf Y}^N$ for $\mathcal V^N(t_0, {\bf Z}^N)$ we find  
\begin{align*}
&\E\left[\mathcal U^N(t_0, {\bf Z}^N)\right] + \E\bigl[ \int_{t_0}^{\tilde \tau^N} \frac{1}{CN} \sum_j |H_p(Y^{N,j}_t, ND_{x^j}\mathcal U^N)-H_p(Y^{N,j}_t, ND_{x^j}\mathcal V^N)|^2dt \bigr] \\ 
& \leq \E\bigl[ \int_{t_0}^{\tilde \tau^N} (\frac1N \sum_j L(Y^{N,j}, -H_p(Y^{N,j}_t, ND_{x^j}\mathcal V^N))+CN^{-1} +\mathcal F(m^N_{{\bf Y}^N_t}))dt 
+\mathcal V^N(\tilde \tau^N, {\bf Y}^N_{\tilde \tau^N})\bigr] + R^N\\ 
& \leq \E\left[\mathcal V^N(t_0, {\bf Z}^N)\right]+ CN^{-1}+ R^N,
\end{align*}
and, using  once more the definition of $R^N$,  we get 
\begin{align*}
& \E\bigl[ \int_{t_0}^{\tilde \tau^N} \frac{1}{CN} \sum_j |H_p(Y^{N,j}_t, ND_{x^j}\mathcal U^N)-H_p(Y^{N,j}_t, ND_{x^j}\mathcal V^N)|^2dt \bigr]  \leq  CN^{-1}+ 2R^N. 
\end{align*}
\end{proof}

\begin{lem}[Convergence of optimal trajectories]\label{lem.CvOptiTraj} For   ${\bf X}^N=(X^{N,i})$ and ${\bf Y}^N=(Y^{N,i})$  defined by \eqref{def.XNi} and \eqref{takis51} respectively, we have
$$
 \E\left[ \sup_{s\in [t_0, \tilde \tau^N]}  N^{-1}\sum_j | X^{N,j}_{s}-Y^{N,j}_{s}| \right]\leq C(N^{-1}+ R^N)^{1/2}
 $$
and
 $$
 \P\left[ \tilde \tau^N<T\right]\leq C (N^{-1/(d+8)}+ (R^N)^{1/2}). 
 $$
\end{lem}

\begin{proof} Lemma~\ref{lem.sldjkfncv}, the regularity of $\mathcal U^N$ in  \eqref{DmULipsch} and an application of Gronwall's inequality  give the first inequality since, for any $t\in [t_0,T]$,  
\begin{align*}
& \E\left[ \sup_{s\in [t_0,t\wedge \tilde \tau^N]}  N^{-1}\sum_j | X^{N,j}_{s}-Y^{N,j}_{s}| \right]\\
& \leq \E\bigl[ \int_{t_0}^{t\wedge \tilde \tau^N}N^{-1} \sum_j |H_p(Y^{N,j}_t, ND_{x^j}\mathcal U^N(t, {\bf Y}^N_t))-H_p(Y^{N,j}_t, ND_{x^j}\mathcal V^N(t, {\bf Y}^N_t))|dt \bigr]\\
& \qquad + 
\E\bigl[ \int_{t_0}^{t\wedge \tilde \tau^N}N^{-1} \sum_j |H_p(Y^{N,j}_t, ND_{x^j}\mathcal U^N(t, {\bf Y}^N_t))-H_p(X^{N,j}_t, ND_{x^j}\mathcal U^N(t, {\bf X}^N_t))|dt \bigr] \\ 
&\leq  C(N^{-1}+ R^N)^{1/2} + C N^{-1}\sum_j \E\left[\int_{t_0}^{t\wedge \tilde \tau^N} 
|X^{N,j}_{s}-Y^{N,j}_{s}|ds\right]. 
\end{align*}
\vs
Then,
\begin{align*}
 \P\left[ \tilde \tau^N<T \right] & \leq  \P\left[ \tau^N<T\right]+ \P\left[ \sup_{s\in [t_0, \tilde \tau^N]}  N^{-1}\sum_j | X^{N,j}_{s}-Y^{N,j}_{s}| >\delta/2\right] \\
&  \leq  C N^{-1/(d+8)} + C\delta^{-1} (N^{-1}+ R^N)^{1/2} .
\end{align*}
\end{proof}

We can proceed now with the proof of the propagation of chaos property. 

\begin{proof}[Proof of Theorem \ref{thm.main20}] It is immediate that 
\be\label{takis52}
\begin{split}
 & \E\left[  \sup_{s\in [t_0, \tilde \tau^N]} {\bf d}_1(m^N_{{\bf Y}^N_s}, m(s)) \right]   \leq
\E\left[  \sup_{s\in [t_0, \tilde \tau^N]} {\bf d}_1(m^N_{{\bf Y}^N_s}, m^N_{{\bf X}^N_s}) \right] +
  \E\left[  \sup_{s\in [t_0, \tilde \tau^N]} {\bf d}_1(m^N_{{\bf X}^N_s}, m(s)) \right] \\
  & \qquad  \leq  \E\left[  \sup_{s\in [t_0, \tilde \tau^N]}N^{-1}\sum_j | X^{N,j}_{s}-Y^{N,j}_{s}| \right] +  \E\left[  \sup_{s\in [t_0, \tilde \tau^N]} {\bf d}_1(m^N_{{\bf X}^N_s}, m(s)) \right] .
\end{split}
\ee
Lemma~\ref{lem.CvOptiTraj} gives that  the first term in the right-hand side of \eqref{takis52}  is not larger than $C(N^{-1}+ R^N)^{1/2} $, where $R^N$ can be estimated  by  Proposition \ref{prop.quantifCV}, while  Lemma \ref{lem.aqemzjlrsdk} implies that  the second term in the right-hand side of \eqref{takis52}  is not larger than $C N^{-1/(d+8)}$.
\vs
Thus,  for some $\gamma'$ depending on $d$ only and  $C$ depending on the initial condition $(t_0,m_0)$,
\begin{align*}
 \E\left[  \sup_{s\in [t_0, \tilde \tau^N]} {\bf d}_1(m^N_{{\bf Y}^N_s}, m(s)) \right]  & \leq CN^{-\gamma'}. 
\end{align*}

Finally, we have 
\begin{align*}
&  \E\left[  \sup_{s\in [t_0, T]} {\bf d}_1(m^N_{{\bf Y}^N_s}, m(s)) \right]  \\
&\qquad  \leq \E\left[  \sup_{s\in [t_0, \tilde \tau^N]} {\bf d}_1(m^N_{{\bf Y}^N_s}, m(s)) \right]  
 +  \E\left[  \sup_{s\in [t_0, T]} {\bf d}_1^2(m^N_{{\bf Y}^N_s}, m(s)) \right]^{1/2} \P\left[ \tilde \tau^N<T\right]^{1/2} \\
 & \qquad \leq CN^{-\gamma'} +C (N^{-1/(d+8)}+ (R^N)^{1/2})^{1/2},
\end{align*}
 the last inequality coming  from Lemma \ref{lem.CvOptiTraj} and the facts  that, since, in view of Lemma \ref{lem.reguum},  $m(s)$ has a uniformly bounded second-order moment  and, by Lemma~\ref{lem.estiVN}, the drift of the process ${\bf Y}^N$ is also uniformly bounded,
$$
 \E\left[  \sup_{s\in [t_0, T]} {\bf d}_1^2(m^N_{{\bf Y}^N_s}, m(s)) \right] \leq C
 $$
 
 Using once more Proposition \ref{prop.quantifCV}, we get that, for a new constant $\gamma''\in (0,1)$ depending on $d$ only and a new constant $C'$ depending on the initial condition $(t_0,m_0)$,
 \begin{align*}
&  \E\left[  \sup_{s\in [t_0, T]} {\bf d}_1(m^N_{{\bf Y}^N_s}, m(s)) \right] \leq C'N^{-\gamma''}.
\end{align*}
\end{proof}

\appendix
\section{The proof of Lemma \ref{lem.secondordercond}} 

\begin{proof} A fact  similar to Lemma~\ref{lem.secondordercond} was  given in \cite{BrCa} for the torus and for smooth initial data. Here, we extend 
the argument for the whole space and general initial conditions and slightly simplify it. 

\vs
We begin with  the existence of a solution to \eqref{eq.forsecondorder}, the uniqueness being obvious in view of the regularity of $\alpha$. 
\vs

Fix $\beta \in C^\infty_c((t_0,T]\times \R^d; \R^d)$ and note that the product $m\beta$ is smooth, because the only singularity of $m$ is at time $t_0$. Thus, there exists a unique classical solution to \eqref{eq.forsecondorder}. 
\vs
In order to prove its regularity, fix $t_0<t_1<t_2$, $\xi\in C^{2+\delta}(\R^d)$, let $w$ be the solution to 
$$
-\partial_t w-\Delta w-\alpha(t,x)\cdot Dw=0\ \ {\rm in}\ \ (t_0,T)\times \R^d \ \ \ w(t_2)=\xi\ \  {\rm in}\ \   \R^d,
$$
and 
note that, for a constant $C$ depending only on the data of the problem, since  the regularity of $\alpha$ depends only on the data of the problem, 
\be\label{abzrsdmcn}
\|w\|_\infty+\|Dw\|_{\infty}\leq C \|\xi\|_{W^{1,\infty}} \ \  {\rm and}\ \  \|w\|_{C^{\delta/2,\delta}}+ \|Dw\|_{C^{\delta/2,\delta}}\leq C\|\xi\|_{C^{2+\delta}}.
\ee

Then, 
$$
\int_{\R^d} \rho(t_2,x)\xi (x)dx = \int_{\R^d} w(t_1,x)\rho(t_1,x)dx- \int_{t_1}^{t_2}\int_{\R^d}\beta(t,x)\cdot Dw(t,x) m(t,dx)dx,
$$
and, choosing $t_1=t_0$ and $t_2$ arbitrary in $[t_0,T]$, we get 
\be\label{zilumjngdrskjdhn}
\sup_{t\in[t_0,T]} \|\rho(t)\|_{(W^{1,\infty})'} \leq   \|\beta(t,\cdot)\|_{L^1_m([0,T]\times \R^d)}. 
\ee
In addition, since, thanks to  \eqref{abzrsdmcn},
  $$\|w(t_1,\cdot)-w(t_2,\cdot)\|_{W^{1,\infty}}\leq C(t_2-t_1)^{\delta/2}\|\xi\|_{C^{2+\delta}}$$ 
 using  \eqref{zilumjngdrskjdhn} we find  
\begin{align*}
& \int_{\R^d} (\rho(t_2,x)-\rho(t_1,x))\xi (x)dx\\
& \qquad  = \int_{\R^d} (w(t_1,x)-w(t_2,x))\rho(t_1,x)dx- \int_{t_1}^{t_2}\int_{\R^d}\beta(t,x)\cdot Dw(t,x) m(t,dx)dx\\
&\qquad \leq  \|w(t_1,\cdot)-w(t_2,\cdot)\|_{W^{1,\infty}}\|\rho(t_1,\cdot)\|_{(W^{1,\infty})'}+C(t_2-t_1)^{1/2}\|\beta\|_{L^2_m([0,T]\times \R^d)}\|Dw\|_\infty\\[1.2mm]
&\qquad \leq  C(t_2-t_1)^{\delta/2}\|\beta(t,\cdot)\|_{L^1_m([0,T]\times \R^d)} \  \|\xi\|_{C^{2+\delta}}+C(t_2-t_1)^{1/2}\|\beta\|_{L^2_m([0,T]\times \R^d)}\|\xi\|_{W^{1,\infty}}. 
\end{align*}
The last estimates  proves the existence of a solution $\rho$ for $\beta \in C^0([t_0,T]\times \R^d; \R^d)$ or for $\beta\in L^\infty$ vanishing near $t=t_0$ by approximation.
\vs

Next, let 
$$
J(m',\alpha') =  \int_{t_0}^T (\int_{\R^d} L(x, \alpha'(t,x))m'(t,dx)+\mathcal F(m'(t))) dt +\mathcal G(m'(T)).
$$
The quantity $J(m',\alpha') $  is defined, for instance, for $m'\in C^0([t_0,T], \Pk)$ and $\alpha'\in C^0([t_0,T]\times \R^d; \R^d)$. Let $\beta\in C^\infty_c((t_0,T]\times \R^d)$ and $\rho$ be the classical solution to \eqref{eq.forsecondorder}, and,  for $h>0$ small, let  $m_h\in C^0([t_0,T],\Pk)$ be  the solution to 
$$
\partial_t m_h -\Delta m_h +{\rm div}(m_h (\alpha+h\beta))=0\qquad {\rm in}\; (t_0,T)\times \R^d \ \ \text{and} \ \ 
m_h(t_0)= m_0 \ \ {\rm in}\ \ \R^d.
$$
Then $m_h=m+h\rho+h^2\xi_h$, where $\xi_h$ solves  in the sense of distribution
$$
\partial_t \xi_h -\Delta \xi_h +{\rm div} (\xi_h (\alpha+h\beta)) +{\rm div}(\beta \rho)= 0\ \ {\rm in}\ \  (t_0,T)\times \R^d \ \ \text{and}\ \ 
\xi_h(t_0)= 0\ \  {\rm in}\ \   \R^d.
$$
The regularity of $\alpha$, $\beta$ and  $\rho$ imply that $\|\xi_h\|_\infty\leq C$, with  $C$ depending  on $\beta$,  and, as $h\to0$,  the $(\xi_h)$s converges weakly in $L^\infty-$weak-$\ast$ to the solution $\xi$ of the same equation with $h=0$. 
\vs
Then 
\begin{align*}
& J(m_h,\alpha+h\beta) =\int_{t_0}^T\left(\int_{\R^d} L(x,\alpha+h\beta)m_h(t,dx)+\mathcal F(m_h(t))\right)dt + \mathcal G(m_h(T))  \\
&\qquad =  J(m,\alpha) + h\Bigl\{ \int_{t_0}^T\Bigl(\int_{\R^d} D_\alpha L(x,\alpha)\cdot \beta(t,x) m(t,dx)+\int_{\R^d} L(x,\alpha)\rho(t,x)dx\\
& \qquad\qquad\qquad +\frac{\delta \mathcal F}{\delta m}(m(t))(\rho(t))\Bigr)dt + \frac{\delta \mathcal G}{\delta m}(m(T))(\rho(T)) \Bigr\} \\
& \qquad + \frac{h^2}{2} \Bigl\{ \int_{t_0}^T\Bigl(\int_{\R^d} D_{\alpha\alpha} L(x,\alpha)\beta(t,x)\cdot \beta(t,x) m(t,dx)+2\int_{\R^d} D_\alpha L(x,\alpha)\cdot \beta(t,x) \rho(t,x)dx\\
& \qquad\qquad\qquad +\int_{\R^d} 2 L(x,\alpha)\xi_h(t,x)dx+2\frac{\delta \mathcal F}{\delta m}(m(t))(\xi_h(t))+ \frac{\delta^2 \mathcal F^2}{\delta m}(m(t))(\rho(t),\rho(t))\Bigr)dt \\
& \qquad \qquad \qquad + 2\frac{\delta \mathcal G}{\delta m}(m(T))(\xi_h(T))+ 
\frac{\delta^2 \mathcal G}{\delta m^2}(m(T))(\rho(T),\rho(T)) \Bigr\} + o(h^2).
\end{align*}
The first-order necessary optimality condition implies that the factor of   $h$ above  vanishes and, therefore, the limit as $h$ vanishes of the term in $h^2$ is nonnegative.
\vs
Thus 
\begin{align*}
& \int_{t_0}^T\Bigl(\int_{\R^d} D_{\alpha\alpha} L(x,\alpha)\beta(t,x)\cdot \beta(t,x) m(t,dx)+2\int_{\R^d} D_\alpha L(x,\alpha)\cdot \beta(t,x) \rho(t,x)dx\\
& \qquad\qquad\qquad +\int_{\R^d} 2 L(x,\alpha)\xi(t,x)dx+2\frac{\delta \mathcal F}{\delta m}(m(t))(\xi(t))+ \frac{\delta^2 \mathcal F^2}{\delta m}(m(t))(\rho(t),\rho(t))\Bigr)dt \\
& \qquad \qquad \qquad + 2\frac{\delta \mathcal G}{\delta m}(m(T))(\xi(T))+ 
\frac{\delta^2 \mathcal G}{\delta m^2}(m(T))(\rho(T),\rho(T)) \; \geq \; 0.
\end{align*}
Using the equation satisfied by the multiplier $u$ and the equation satisfied by $\xi$ we find  
\begin{align*}
& \int_{t_0}^T \int_{\R^d} (L(x,\alpha)\xi(t,x)dx+\frac{\delta \mathcal F}{\delta m}(m(t))(\xi(t)))dt+ \frac{\delta \mathcal G}{\delta m}(m(T))(\xi(T))\\
& \qquad = \int_{t_0}^T \int_{\R^d} ((-H(x,Du)-\alpha\cdot Du)\xi(t,x)dx+\frac{\delta \mathcal F}{\delta m}(m(t))(\xi(t)))dt+ \frac{\delta \mathcal G}{\delta m}(m(T))(\xi(T)) \\
&\qquad=  -\int_{t_0}^T\int_{\R^d} Du(t,x)\cdot \beta(t,x) \rho(t,x) dxdt = - \int_{t_0}^T\int_{\R^d} D_\alpha L(x,\alpha)\cdot \beta(t,x) \rho(t,x)dxdt. 
\end{align*}
Inserting  the last  equality in  the previous inequality yields the second-order optimality condition when $\beta$ is smooth. The general case is obtained by approximation using the estimates in the first part of the proof.
\end{proof}

\bibliographystyle{siam}

\bigskip

%\author[P. Cardaliaguet]{Pierre Cardaliaguet}
\noindent ($^{1}$) Ceremade (UMR CNRS 7534), Universit\'{e} Paris-Dauphine PSL\\ Place du Mar\'{e}chal De Lattre De Tassigny 75775 Paris CEDEX 16, France\\
email: cardaliaguet@ceremade.dauphine.fr
\\ \\
%\author[P.~E. Souganidis]{Panagiotis E. Souganidis}
\noindent ($^{2}$) Department of Mathematics, 
The University of Chicago, \\
5734 S. University Ave., 
Chicago, IL 60637, USA  \\ 
email: souganidis@math.uchicago.edu

\end{document}